\documentclass[12pt]{article}
\usepackage{lineno}  
\usepackage{here}
\usepackage{natbib}
\usepackage{lscape}
\usepackage{longtable}
\usepackage{bm,graphicx,graphics,url,psfrag}
\usepackage{lipsum} % just for dummy text- not needed for a longtable

\usepackage{amsmath,amsthm,natbib,amssymb,amsfonts,mathrsfs}
\usepackage{epsf,epsfig,graphicx,psfrag,color}
\usepackage{mathtools,bm,here}
\usepackage{breakcites}
\usepackage{subfigure}
\usepackage{comment}

\newtheorem{theorem}{Theorem}[section]

\newtheorem{lem}[theorem]{Lemma}
\newtheorem*{rem}{Remark}
\newtheorem{prop}{Proposition}

\newtheorem{exmp}{Example}[section]

\usepackage{breakcites}

\usepackage[paperwidth=28cm]{geometry}

\usepackage{multirow}

\usepackage{listings}
\usepackage{color}

\definecolor{codegreen}{rgb}{0,0.6,0}
\definecolor{codegray}{rgb}{0.5,0.5,0.5}
\definecolor{codepurple}{rgb}{0.58,0,0.82}
\definecolor{backcolour}{rgb}{0.95,0.95,0.92}

\lstdefinestyle{mystyle}{
    backgroundcolor=\color{backcolour},
    commentstyle=\color{codegreen},
    keywordstyle=\color{magenta},
    numberstyle=\tiny\color{codegray},
    stringstyle=\color{codepurple},
    basicstyle=\footnotesize,
    breakatwhitespace=false,
    breaklines=true,
    captionpos=b,
    keepspaces=true,
    numbers=left,
    numbersep=5pt,
    showspaces=false,
    showstringspaces=false,
    showtabs=true,
    tabsize=2
}

\title{On a bimodal Birnbaum-Saunders distribution with applications to lifetime data}

\author{Roberto Vila$^{1}$,
Jeremias Le\~ao$^{2}$,
Helton Saulo$^{1}$,
Mirza Nabeed$^{3}$, and Manoel Santos-Neto$^{4,5}$\\
 $^{1}$Department of Statistics, University of Bras\'ilia, Bras\'ilia, BRA\\
    $^{2}$Department of Statistics, Federal University of Amazonas, Manaus, BRA\\
     $^{3}$Department of Statistics, University of Gujrat, Gujrat, PAK\\
    $^{4}$Department of Statistics, Federal University of S\~ao Carlos, S\~ao Carlos, BRA\\
    $^{5}$Department of Statistics, Federal University of Campina Grande, Campina Grande, BRA
}

\begin{document}
%\linenumbers
%\listofchanges
\maketitle

\begin{abstract}
The Birnbaum-Saunders distribution is a flexible and useful 
model which has been used in several fields. In this paper, 
a new bimodal version of this distribution based on the 
alpha-skew-normal distribution is established. We discuss some of its 
mathematical and inferential properties. We consider likelihood-based 
methods to estimate the model parameters. We carry out a Monte Carlo 
simulation study to evaluate the performance of the maximum likelihood estimators. 
For illustrative purposes, three real data sets 
are analyzed. The results indicated that the proposed model outperformed some 
existing models in the literature, in special, a recent 
bimodal extension of the Birnbaum-Saunders distribution.

%\begin{classcode}
%Primary 62F99; 62E25. \quad  Secondary  62N01.
%\end{classcode}

\end{abstract}

\section{Introduction}\label{sec:1}
Despite its broad applicability in many fields, see, for example, \cite{bll:07}, \cite{b:10}, 
\cite{vslc:10}, \cite{plbl:12}, \cite{slzm:13},\cite{lmsar:14,lsml:14}, \cite{l:16} and \cite{llst:17}, the Birnbaum-Saunders (BS) 
distribution \citep{bs:69} is not suitable to model 
bimodal data. This distribution is positively skewed with positive support and is related to 
the normal distribution through the stochastic representation
\begin{equation}\label{sec1:main}
	T=\frac{\beta}{4} \left[\alpha{Z} + \sqrt{(\alpha{Z})^2+4}\right]^{2},
\end{equation}
where $T \sim \rm{BS}(\alpha, \beta)$, $Z \sim\mbox{N}(0, 1)$ and $\alpha>0$, $\beta>0$ are shape and scale
parameters, respectively. The $\rm{BS}(\alpha, \beta)$ probability 
density function (PDF) and cumulative distribution function (CDF) are respectively given by
\begin{equation}\label{sec1:02}
	f(t;\alpha,\beta)=\phi(a(t))\,
	\frac{t^{-3/2}(t+\beta)}{2\alpha\,\beta^{1/2}}
	\quad  \mbox{and} \quad
	F(t;\alpha,\beta)=\Phi(a(t)), \quad t>0,
\end{equation}
where $\phi(\cdot)$ and $\Phi(\cdot)$ are standard normal PDF and CDF, respectively, and
\begin{align}\label{at}
	a(t) = \frac{1}{\alpha} \left[\sqrt{\frac{t}{\beta}}-\sqrt{\frac{\beta}{t}}\right].
\end{align}
Note that the $k$-th derivative of $a(t)$, denoted by $a^{(k)}(t)$, satisfies
$a^{(k)}(t)>0$ (or $<0$) for $k$ odd (or $k$ even), where $k\geqslant 1$. 
Some special cases of these derivatives are
\begin{align}\label{derivates}
	a'(t)=
	{1\over 2\alpha t}
	\left[\sqrt{t\over \beta}+{\sqrt{\beta\over t}}\right],
	\quad 
	a''(t)=
	-{1\over 4\alpha t^2}
	\left[\sqrt{t\over\beta}+3{\sqrt{\beta\over t}}\right]
	\quad \text{and}\quad 
	a'''(t)={3\over 8\alpha t^3}\left[\sqrt{t\over\beta}+5{\sqrt{\beta\over t}}\right].
\end{align}
Note also that the function $a(\cdot)$ has inverse specified by \eqref{sec1:main}.
In order not to cause confusion, hereafter we will write $a^{-\perp}(\cdot)$ to denote the inverse of 
function $a(\cdot)$.

The stochastic representation in \eqref{sec1:main} allows us to obtain several ge\-neralizations of the BS model. For example, 
\cite{dl:05} assumed that ${Z}$ follows a standard symmetric distribution in the real line and obtained the class
of generalized BS distributions. On the same line, \cite{blsv:09} proposed scale-mixture BS distributions by assuming 
that $Z$ belongs to the family of scale mixture of normal distributions. Many other generalizations can be obtained in order to obtain 
a new distribution with domain on the positive numbers; see \cite{l:16}.

In general, one uses mixtures of distributions for describing bimodal data. However, it may be troublesome as identifiability problems may arise in the parameter estimation of the model; see \cite{llh:07,lly:07} and \cite{ges:11}. In this sense, new mixture-free models which have the capacity to accommodate unimodal and bimodal data are very important. Some 
asymmetric bimodal models in the real line have been discussed by \cite{ac:03}, \cite{k:05}, and \cite{mg:04}, among others. {In the context of bimodal BS models, \cite{bgkls:11} introduced a mixture distribution of two different BS models (MXBS) and studied its characteristics. On the other hand, \cite{omb:17} introduced a bimodal extension of the BS distribution, denoted by BBSO, based on the approach described in \cite{ges:11}.
	{In addition, the authors also studied} the probabilistic properties and moments of the BBSO distribution, and showed that this model can fit well both unimodal and bimodal data in comparison with the BS, 
	log-normal and skew-normal BS models. A thorough inference study on the parameters that index the BBSO distribution was addressed by \cite{fc:16}.}

%\cite{omb:17} studied, amongst other things, the probabilistic properties 
%and moments of the BBSO distribution, and showed that this model can fit well both unimodal and bimodal 
%data in comparison with the BS, 
%log-normal and skew-normal BS models.

In this paper, we introduce a new bimodal version of the BS distribution, denoted by BBS, 
by assuming that $Z$ in \eqref{sec1:main} 
follows an alpha-skew-normal (ANS) distribution 
discussed by \cite{e:10}. We present a 
statistical methodology based on the proposed BBS distribution including model formulation, mathematical properties, estimation and inference 
based on the maximum likelihood (ML) method. We evaluate 
the performance of the ML estimators by Monte Carlo (MC) simulations. { Three 
	real data illustrations indicated that the proposed BBS model provides better adjustment compared to the BBSO model proposed by \cite{omb:17}. The proposed BBS distribution has some advantages over existing bimodal BS models: (i) unlike the BBSO model, the proposed BBS distribution does not suffer from convergence problems in the {optimization} process of the profile log-likelihood function as pointed out by \cite{fc:16}; (ii) the proposed BBS distribution does not present {identifiability} problems commonly encountered in mixture models, such as the MXBS distribution; and (iii) the proposed model does not present label switching problems \citep{celeuxetal:06}, that is, in a bimodal context with two groups, during the estimation an individual who was in group B can incorrectly stay in A and vice versa.}

The rest of the paper proceeds as follows. In Section \ref{sec:2},
we introduce the BBS distribution and discuss some related results. In Section~\ref{sec:3}, we consider likelihood-based methods to estimate the model parameters and to perform inference. In Section~\ref{sec:4}, we carry out a  MC simulation study to evaluate the performance of the ML estimators. 
In Section~\ref{sec:5}, we illustrate the proposed methodology with three real data sets. Finally, 
in Section~\ref{sec:6}, we make some concluding remarks and discuss future research.

%\section{The bimodal Birnbaum-Saunders distribution}\label{sec:2}
\section{The BBS distribution}\label{sec:2}
If a random variable (RV) $X$ has an ASN distribution with 
parameter $\delta$, denoted by $X\sim\text{ASN}(\delta)$, 
then its PDF and CDF are given by 
\begin{equation}\label{sec2:01}
	g(x)=\frac{(1-\delta{x})^2+1}{2+\delta^2}\,\phi(x)\,\, \mbox{and} \,\, 
	G(x)=\Phi(x)+\delta\left(\frac{2-\delta{x}}{2+\delta^2}\right)\phi(x),%, \quad x\in\mathbb{R},
\end{equation}
where $x,\delta\in\mathbb{R}$ and $\delta$ is an asymmetric parameter 
that controls the uni-bimodality effect; see \cite{e:10}. The PDF of the BS distribution, based on the alpha-skew-normal model, is given by
\begin{equation}\label{sec2:02}
	f(t;\alpha,\beta,\delta)
	=
	\frac{(1-\delta{a(t)})^2+1}{2+\delta^2} \, \phi(a(t)) \,
	\frac{t^{-3/2}(t+\beta)}{2\alpha\,\beta^{1/2}},
	\quad t>0,
\end{equation}
where $a(\cdot)$ is as in \eqref{at} and the notation 
$T\sim\text{BBS}(\alpha,\beta,\delta)$ is used. 
If $\delta=0$, then the classical
$\text{BS}(\alpha,\beta)$ distribution is obtained. The 
corresponding $\text{BBS}(\alpha,\beta,\delta)$ CDF is given by
\begin{equation}\label{sec2:03}
	F(t;\alpha,\beta,\delta)
	=
	\Phi(a(t))
	+
	\delta\left(\frac{2-\delta{a(t)}}{2+\delta^2}\right) \phi(a(t)),
	\quad t>0.
\end{equation}
Note that $f(t;\alpha,\beta,\delta)=g(a(t))a'(t)=(G\circ a)'(t)$,
$F(t;\alpha,\beta,\delta)=(G\circ a)(t)$ and
$\lim_{\delta\to\pm\infty}\{F(t;\alpha,\beta,\delta)+\phi(a(t))a(t)\}=F(t;\alpha,\beta)$.

Differentiating  the PDF of the BBS distribution \eqref{sec2:02} we obtain
\begin{align}
	f'(t;\alpha,\beta,\delta)&=g'(a(t))[a'(t)]^2+g(a(t))a''(t) \quad \text{and} \label{1-derivate}
	\\[0,2cm]	
	f''(t;\alpha,\beta,\delta)&= g''(a(t))[a'(t)]^3+3g'(a(t))a'(t)a''(t)+g(a(t))a'''(t), \label{2-derivate}
\end{align}
where
$
g''(x)
=
-xg'(x)
+
\phi(x)
\left(-3\delta^2 x^2 + 4\delta x - 2(1-\delta^2)\right)/(2+\delta^2)
$ 
and 
\[
g'(x)=
{\phi(x)\over 2+\delta^2}
\left(
-\delta^2 x^3+2\delta x^2-2(1-\delta^2)x-2\delta
\right).
\]

The survival and hazard functions, denoted by SF and HR, respectively, of the BBS distribution are given by 
$
S(t;\alpha,\beta,\delta)= 1-(G\circ a)(t)
$
and
\begin{equation*}\label{sec3:01}
	%\quad 
	%\mbox{and}
	%\quad 
	h(t;\alpha,\beta,\delta)
	=
	\frac{f(t;\alpha,\beta,\delta)}{1-F(t;\alpha,\beta,\delta)}
	=
	{(G\circ a)'(t)\over S(t;\alpha,\beta,\delta)},
	\quad t>0,
\end{equation*}
respectively. From Figure \ref{fig:1} we note some different shapes of the BBS PDF for different combinations of parameters. These figures reveal clearly the bimodality effect caused by the parameter $\delta$. Also, Figure~\ref{fig:3hazards} shows unimodal and bimodal shapes for the BBS HR.

\vspace{-0.2cm}
\begin{figure}[h!]
	\vspace{-0.25cm}
	\centering
	\psfrag{0}[c][c]{\scriptsize{0}}
	\psfrag{1}[c][c]{\scriptsize{1}}
	\psfrag{2}[c][c]{\scriptsize{2}}
	\psfrag{3}[c][c]{\scriptsize{3}}
	\psfrag{4}[c][c]{\scriptsize{4}}
	\psfrag{5}[c][c]{\scriptsize{5}}
	\psfrag{6}[c][c]{\scriptsize{6}}
	\psfrag{8}[c][c]{\scriptsize{8}}
	\psfrag{10}[c][c]{\scriptsize{10}}
	\psfrag{12}[c][c]{\scriptsize{12}}
	\psfrag{15}[c][c]{\scriptsize{15}}
	\psfrag{20}[c][c]{\scriptsize{20}}
	\psfrag{25}[c][c]{\scriptsize{25}}
	\psfrag{30}[c][c]{\scriptsize{30}}
	\psfrag{0.00}[c][c]{\scriptsize{0.00}}
	\psfrag{0.05}[c][c]{\scriptsize{0.05}}
	\psfrag{0.10}[c][c]{\scriptsize{0.10}}
	\psfrag{0.15}[c][c]{\scriptsize{0.15}}
	\psfrag{0.20}[c][c]{\scriptsize{0.20}}
	\psfrag{0.25}[c][c]{\scriptsize{0.25}}
	\psfrag{0.30}[c][c]{\scriptsize{0.30}}
	\psfrag{0.0}[c][c]{\scriptsize{0.0}}
	\psfrag{0.2}[c][c]{\scriptsize{0.2}}
	\psfrag{0.3}[c][c]{\scriptsize{0.3}}
	\psfrag{0.4}[c][c]{\scriptsize{0.4}}
	\psfrag{0.5}[c][c]{\scriptsize{0.5}}
	\psfrag{0.6}[c][c]{\scriptsize{0.6}}
	\psfrag{0.8}[c][c]{\scriptsize{0.8}}
	\psfrag{1.0}[c][c]{\scriptsize{1.0}}
	\psfrag{1.2}[c][c]{\scriptsize{1.2}}
	\psfrag{1.5}[c][c]{\scriptsize{1.5}}
	\psfrag{2.0}[c][c]{\scriptsize{2.0}}
	\psfrag{2.5}[c][c]{\scriptsize{2.5}}
	\psfrag{3.0}[c][c]{\scriptsize{3.0}}
	\psfrag{de}[c][c]{\scriptsize{PDF}}
	\psfrag{xx}[c][c]{\scriptsize{$t$}}
	
	\psfrag{a}[l][l]{\tiny{$\alpha=0.10,\delta=-10$}}
	\psfrag{b}[l][l]{\tiny{$\alpha=0.25,\delta=-10$}}
	\psfrag{c}[l][l]{\tiny{$\alpha=0.50,\delta=-10$}}
	\psfrag{d}[l][l]{\tiny{$\alpha=1.00,\delta=-10$}}
	\psfrag{e}[l][l]{\tiny{$\alpha=1.50,\delta=-10$}}
	
	\psfrag{g}[l][c]{\tiny{$\alpha=0.10,\delta=5$}}
	\psfrag{h}[l][c]{\tiny{$\alpha=0.25,\delta=5$}}
	\psfrag{i}[l][c]{\tiny{$\alpha=0.50,\delta=5$}}
	\psfrag{j}[l][c]{\tiny{$\alpha=1.00,\delta=5$}}
	\psfrag{k}[l][c]{\tiny{$\alpha=1.50,\delta=5$}}
	
	\psfrag{m}[l][c]{\tiny{$\alpha=0.10,\delta=-2$}}
	\psfrag{n}[l][c]{\tiny{$\alpha=0.25,\delta=-2$}}
	\psfrag{o}[l][c]{\tiny{$\alpha=0.50,\delta=-2$}}
	\psfrag{p}[l][c]{\tiny{$\alpha=1.00,\delta=-2$}}
	\psfrag{q}[l][c]{\tiny{$\alpha=1.50,\delta=-2$}}
	
	\psfrag{s}[l][c]{\tiny{$\alpha=0.10,\beta=3,\delta=10$}}
	\psfrag{t}[l][c]{\tiny{$\alpha=0.25,\beta=3,\delta=10$}}
	\psfrag{u}[l][c]{\tiny{$\alpha=0.50,\beta=3,\delta=10$}}
	\psfrag{v}[l][c]{\tiny{$\alpha=1.00,\beta=3,\delta=10$}}
	\psfrag{x}[l][c]{\tiny{$\alpha=1.50,\beta=3,\delta=10$}}

	{\includegraphics[height=6.2cm,width=4.2cm,angle=-90]{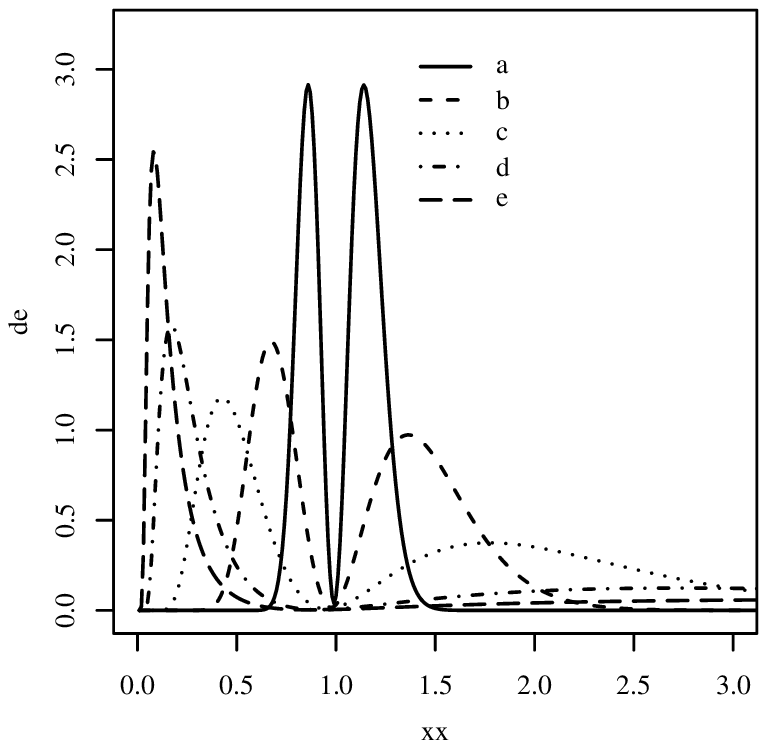}}
	{\includegraphics[height=6.2cm,width=4.2cm,angle=-90]{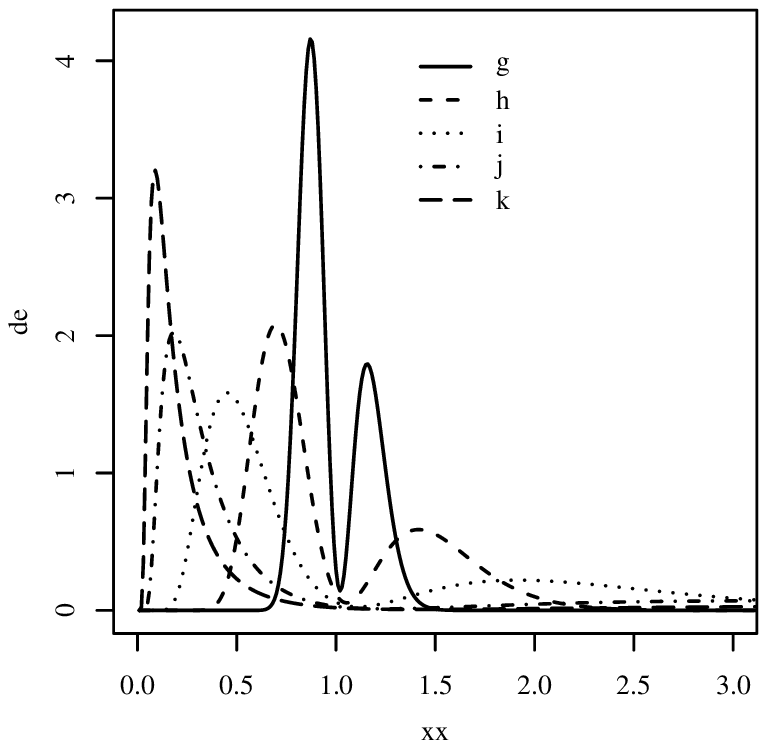}}\\
	{\includegraphics[height=6.2cm,width=4.2cm,angle=-90]{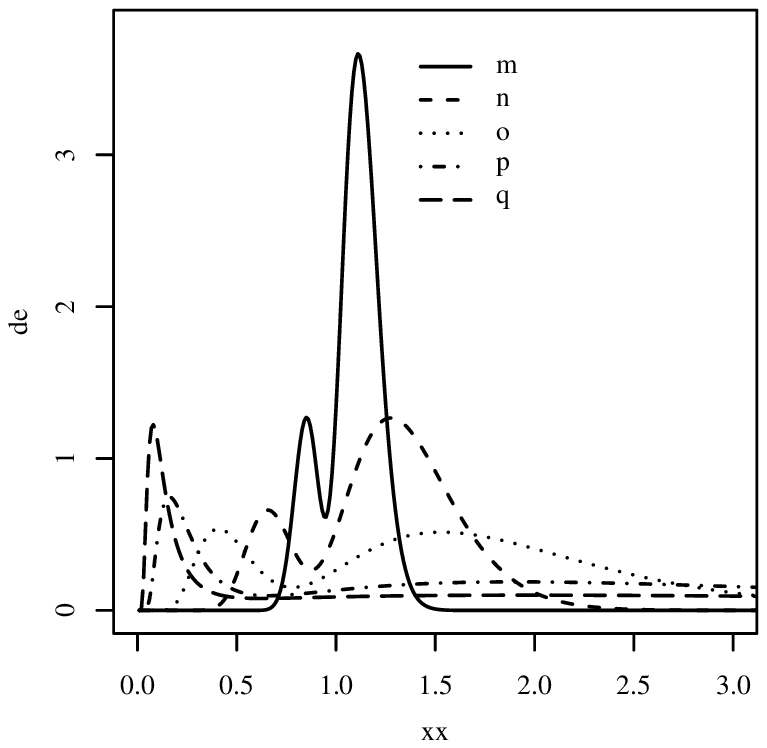}}
	{\includegraphics[height=6.2cm,width=4.2cm,angle=-90]{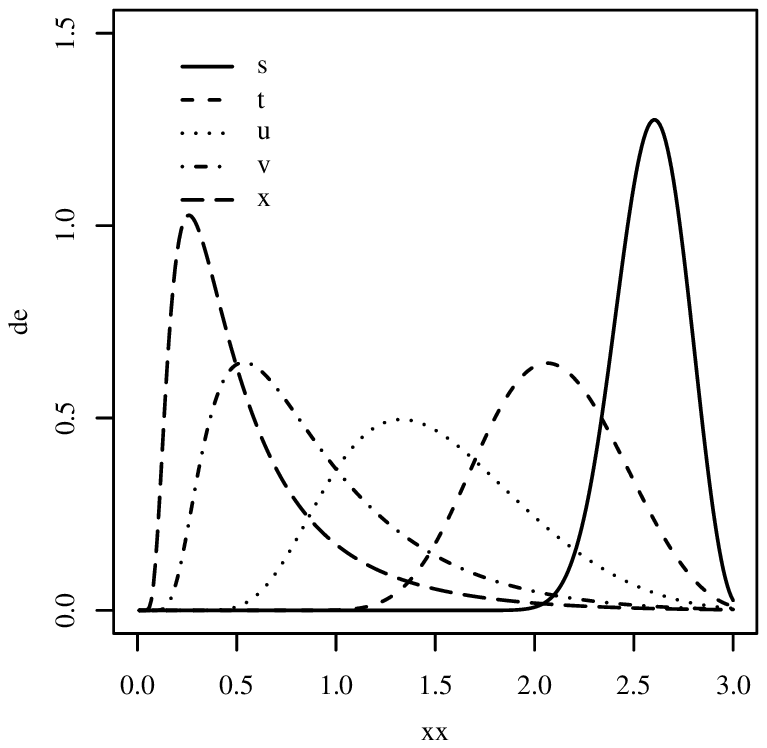}}

	\vspace{-0.2cm}
	\caption{BBS PDFs for some parameter values ($\beta=1.0$).}\label{fig:1}
\end{figure}

\subsection{Some properties of the BBS distribution} \label{sec:prop}

%% begin color red %%%%%%%%%%%%%%%%%%%%%%%%%%%%%%%%%%%%%%%%%%%%%%%%%%%%%%%%%%%%%%%%%%%%%%%%%%%%%%%%%%%%%%%%%%%%%%%%%%%%%%%%%%%%%%%%

\begin{lem}\label{lemma-mono}
	Let $t_0=\beta [(\alpha/\delta)+\sqrt{(\alpha/\delta)^2+4}]^2/4$.
	The PDF of the BBS distribution $t\mapsto f(t;\alpha,\beta,\delta)$ 
	defined in \eqref{sec2:02}
	is a decreasing function when
	\begin{enumerate}
		\item $\delta= 0$ and $t>\beta$, or
		\item $t<t_0$ ($t>t_0$), 
		for each $\delta>0$ ($\delta<0$).
	\end{enumerate}
	\begin{proof}
		1. If $\delta=0$, we have $f(t;\alpha,\beta,\delta)=f(t;\alpha,\beta)=\phi(a(t))a'(t)$, $t>0$.
		Since $a'(t)>0$ and $a''(t)<0$ (see \eqref{derivates}) we have that 
		$t\mapsto\phi(a(t))$ and $t\mapsto a'(t)$ are
		decreasing functions whenever $t>\beta$. Therefore, since 
		the PDF of the BBS distribution is a product of nonnegative decreasing functions, it is 
		decreasing for each $t>\beta$.
		
		2. Let
		$r(x) =((1-\delta{x})^2+1)/(2+\delta^2)$, $x\in\mathbb{R}$.
		Note that  $(r\circ a)'(t)=-2\delta(1-\delta a(t)) a'(t)/(2+\delta^2)$, $t>0$.
		If $\delta>0$ ( $\delta<0$ ) then $(r\circ a)'(t)<0$ whenever $t<t_0$ ( $t>t_0$ ). 
		Since, by Item 1, the function $t\mapsto f(t;\alpha,\beta)$ is decreasing and
		$f(t;\alpha,\beta,\delta)=r(a(t))f(t;\alpha,\beta)$, $t>0$, the proof follows.
	\end{proof}
\end{lem}
\begin{prop}\label{prop:01}
	Let $T\sim\text{BBS}(\alpha,\beta,\delta)$ as defined in \eqref{sec2:01}. Then, 
	\begin{enumerate}
		\item $a(T)\sim\text{ASN}(\delta)$;
		\item $cT\sim\text{BBS}(\alpha,c\beta,\delta)$, with $c>0$;
		\item $T^{-1}\sim\text{BBS}(\alpha,\beta^{-1},-\delta)$.
	\end{enumerate}
\end{prop}
\begin{proof}
	Since 
	$\mathbb{P}(a(T)\leqslant t)=F(a^{-\perp}(t);\alpha,\beta,\delta)$, we have 
	that the 
	PDF
	of $a(T)$ is equal to 
	$
	{f(a^{-\perp}(t);\alpha,\beta,\delta)}=g(t) a'(a^{-\perp}(t)).
	$
	Then
	$a(T)\sim\text{ASN}(\delta)$. 
	The proof of the Items 2 and 3 are immediate, 
	after making convenient variables transformations.
\end{proof}

\begin{figure}[h!]
	\vspace{-0.25cm}
	\centering
	\psfrag{0}[c][c]{\scriptsize{0}}
	\psfrag{1}[c][c]{\scriptsize{1}}
	\psfrag{2}[c][c]{\scriptsize{2}}
	\psfrag{3}[c][c]{\scriptsize{3}}
	\psfrag{4}[c][c]{\scriptsize{4}}
	\psfrag{5}[c][c]{\scriptsize{5}}
	\psfrag{6}[c][c]{\scriptsize{6}}
	\psfrag{8}[c][c]{\scriptsize{8}}
	\psfrag{10}[c][c]{\scriptsize{10}}
	\psfrag{12}[c][c]{\scriptsize{12}}
	\psfrag{15}[c][c]{\scriptsize{15}}
	\psfrag{20}[c][c]{\scriptsize{20}}
	\psfrag{25}[c][c]{\scriptsize{25}}
	\psfrag{30}[c][c]{\scriptsize{30}}
	\psfrag{0.00}[c][c]{\scriptsize{0.00}}
	\psfrag{0.05}[c][c]{\scriptsize{0.05}}
	\psfrag{0.10}[c][c]{\scriptsize{0.10}}
	\psfrag{0.15}[c][c]{\scriptsize{0.15}}
	\psfrag{0.20}[c][c]{\scriptsize{0.20}}
	\psfrag{0.25}[c][c]{\scriptsize{0.25}}
	\psfrag{0.30}[c][c]{\scriptsize{0.30}}
	\psfrag{0.0}[c][c]{\scriptsize{0.0}}
	\psfrag{0.2}[c][c]{\scriptsize{0.2}}
	\psfrag{0.3}[c][c]{\scriptsize{0.3}}
	\psfrag{0.4}[c][c]{\scriptsize{0.4}}
	\psfrag{0.5}[c][c]{\scriptsize{0.5}}
	\psfrag{0.6}[c][c]{\scriptsize{0.6}}
	\psfrag{0.8}[c][c]{\scriptsize{0.8}}
	\psfrag{1.0}[c][c]{\scriptsize{1.0}}
	\psfrag{1.2}[c][c]{\scriptsize{1.2}}
	\psfrag{1.5}[c][c]{\scriptsize{1.5}}
	\psfrag{2.0}[c][c]{\scriptsize{2.0}}
	\psfrag{2.5}[c][c]{\scriptsize{2.5}}
	\psfrag{3.0}[c][c]{\scriptsize{3.0}}
	\psfrag{HF}[c][c]{\scriptsize{HR}}
	\psfrag{t}[c][c]{\scriptsize{$t$}}
	
	\psfrag{a}[l][c]{\tiny{$\alpha=0.10,\delta=10$}}
	\psfrag{b}[l][c]{\tiny{$\alpha=0.25,\delta=10$}}
	\psfrag{c}[l][c]{\tiny{$\alpha=0.50,\delta=10$}}
	\psfrag{d}[l][c]{\tiny{$\alpha=1.00,\delta=10$}}
	\psfrag{e}[l][c]{\tiny{$\alpha=1.50,\delta=10$}}

	\psfrag{g}[l][c]{\tiny{$\alpha=0.10,\delta=0.1$}}
	\psfrag{h}[l][c]{\tiny{$\alpha=0.25,\delta=0.1$}}
	\psfrag{i}[l][c]{\tiny{$\alpha=0.50,\delta=0.1$}}
	\psfrag{j}[l][c]{\tiny{$\alpha=1.00,\delta=0.1$}}
	\psfrag{k}[l][c]{\tiny{$\alpha=1.50,\delta=0.1$}}
	
	\psfrag{m}[l][c]{\tiny{$\alpha=0.5,\delta=-2$}}
	\psfrag{n}[l][c]{\tiny{$\alpha=0.5,\delta=1$}}
	\psfrag{o}[l][c]{\tiny{$\alpha=0.5,\delta=0$}}
	\psfrag{p}[l][c]{\tiny{$\alpha=0.5,\delta=1$}}
	\psfrag{q}[l][c]{\tiny{$\alpha=0.5,\delta=2$}}
	
	\psfrag{s}[l][c]{\tiny{$\alpha=0.1,\beta=0.5,\delta=2$}}
	\psfrag{t}[l][c]{\tiny{$\alpha=0.1,\beta=0.8,\delta=2$}}
	\psfrag{u}[l][c]{\tiny{$\alpha=0.1,\beta=1.0,\delta=2$}}
	\psfrag{v}[l][c]{\tiny{$\alpha=0.1,\beta=1.2,\delta=2$}}
	\psfrag{x}[l][c]{\tiny{$\alpha=0.1,\beta=1.5,\delta=2$}}
	
	{\includegraphics[height=6.2cm,width=4.2cm,angle=-90]{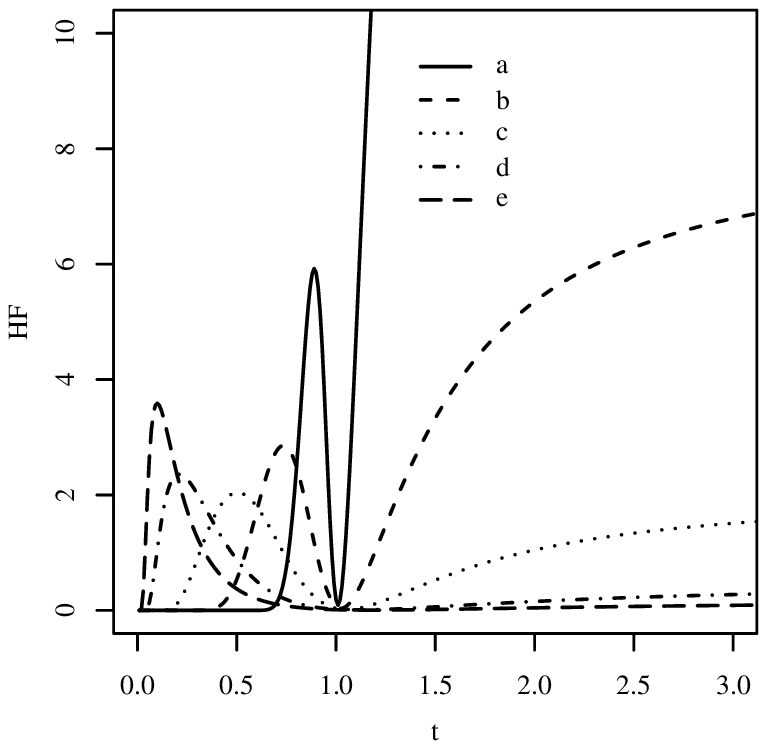}}
	{\includegraphics[height=6.2cm,width=4.2cm,angle=-90]{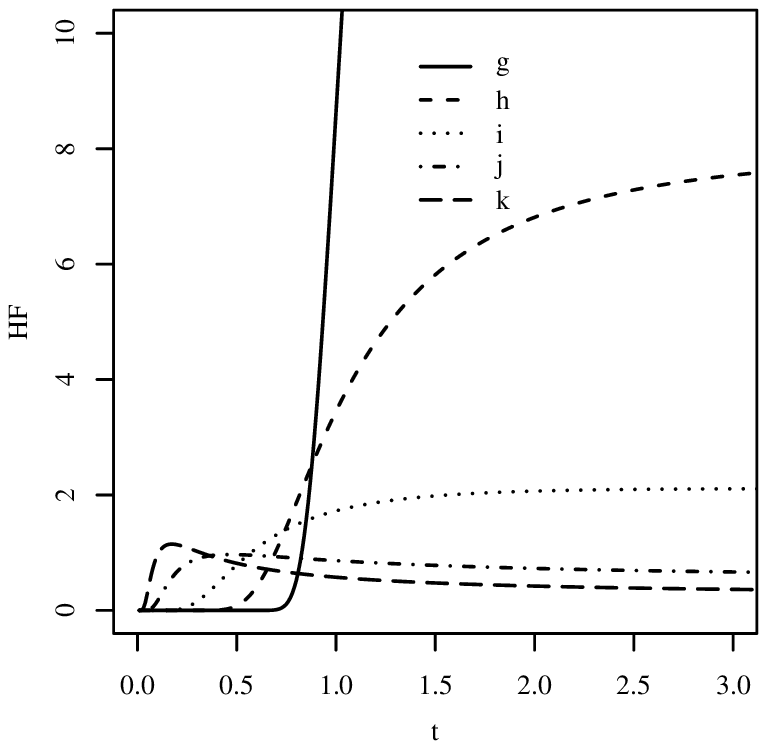}}
	{\includegraphics[height=6.2cm,width=4.2cm,angle=-90]{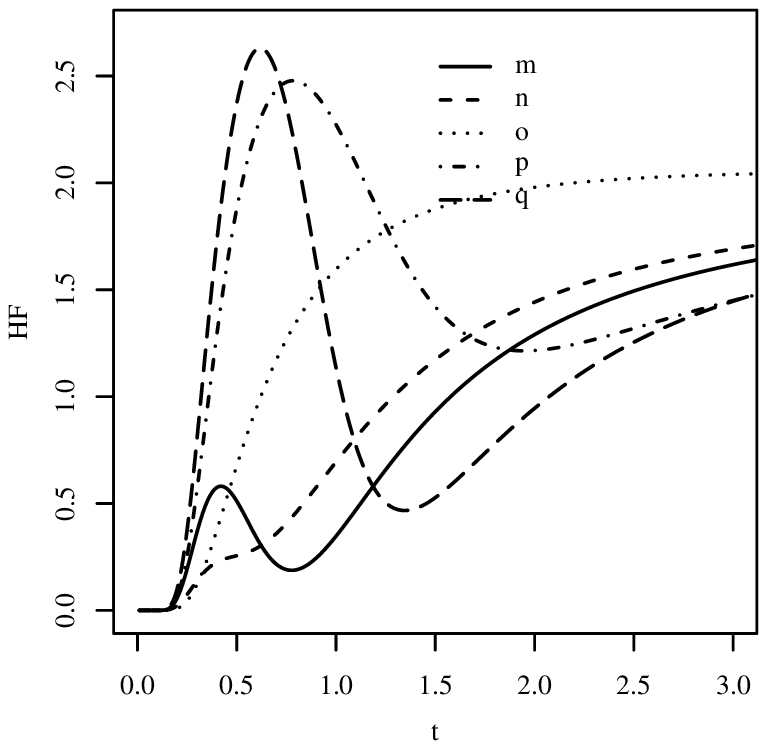}}
	{\includegraphics[height=6.2cm,width=4.2cm,angle=-90]{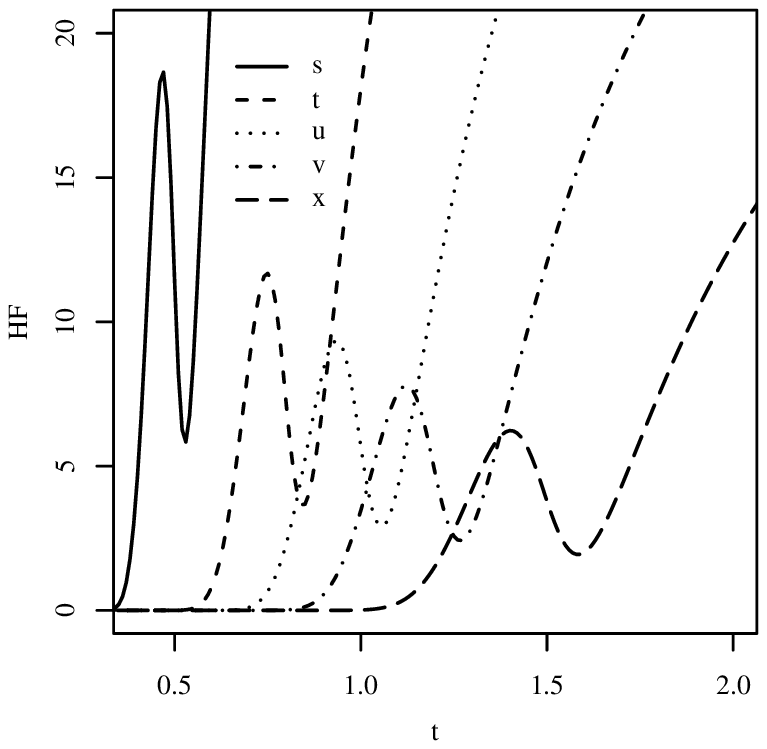}}
	\vspace{-0.2cm}
	\caption{BBS HRs for some parameter values ($\beta=1.0$).}\label{fig:3hazards}
\end{figure}
\begin{prop}\label{prop-exp}
	Let $T\sim\mbox{BBS}(\alpha,\beta,\delta)$ and $X \sim \mbox{ASN}(\delta)$ 
	and suppose $\mathbb{E}[T^{n}]$ exists, $n\geqslant 1$. Then, we have
	\begin{enumerate}
		\item 
		$
		\mathbb{E}[T] 
		=
		{\beta\over 2}
		\left(-\alpha\,{2\delta\over 2+\delta^2}+{\omega}_{0,1}\right);
		$
		\item 
		$
		\mathbb{E}[T^2]
		=
		\left({\beta\over 2}\right)^2
		\left(4+2\alpha^2\,{2+3\delta^2\over 2+\delta^2}+2\alpha\omega_{1,1} \right);
		$
		\item 
		$
		\mathbb{E}[T^3]
		=
		\left({\beta\over 2}\right)^3
		\left(
		-24\alpha(\alpha^2+1)\,{\delta\over 2 +\delta^2}
		+3\alpha^2\omega_{2,1} +\alpha\omega_{1,1}
		+
		\omega_{0,3}
		\right);
		$
		\item 
		$
		\mathbb{E}[T^4]
		=
		\left({\beta\over 2}\right)^4
		\left(
		16+{8\alpha^2}\,{8+12\delta^2+6\alpha^2+15\delta^2\alpha^2 \over 2+\delta^2}+
		4\alpha(\alpha^2\omega_{3,1}+\omega_{1,3})
		\right)
		$
		\text{and}
		\item
		$
		\mathrm{Var}[T]
		=
		\left({\beta\over 2}\right)^2
		\left(
		2\alpha^2\,{4+6\delta^2+3\delta^4\over (2+\delta^2)^2} 
		+
		\omega_{0,1}^2+2\alpha (\omega_{1,1}+\omega_{0,1}\,{2\delta\over 2+\delta^2})
		\right),
		$
	\end{enumerate}
	where
	${\omega}_{r,k}=\mathbb{E}[X^{r}(\sqrt{\alpha^{2}X^{2}+4})^k]$.
\end{prop}
\begin{proof}
	By Proposition \ref{prop:01} Item 1 we have $a(T)\sim\text{ASN}(\delta)$ which implies that
	$\mathbb{E}[T^{n}]
	=
	\mathbb{E}\left[\{a^{-\perp}(X)\}^n\right].
	$
	Then, the proof is immediate since
	\begin{align}\label{expct}
		%\mathbb{E}[T^{n}]
		%=
		\mathbb{E}\left[\{a^{-\perp}(X)\}^n\right]
		=
		\left(\beta\over 2\right)^n
		\mathbb{E}\left[\left(\alpha X+\sqrt{(\alpha X)^2+4}\right)^n\right],
		\quad n\geqslant 1
	\end{align}
	and
	$
	\mathbb{E}[X]=-{2\delta/ (2+\delta^2)},
	\,
	\mathbb{E}[X^2]=1-\delta\mathbb{E}[X],
	\,
	\mathbb{E}[X^3]=3\mathbb{E}[X],
	$
	and
	$\mathbb{E}[X^4]=3(1-2\delta\mathbb{E}[X]).$
\end{proof}
\begin{rem}\label{exist-vl}
	By using the Binomial Theorem and \eqref{expct}  
	note that $\mathbb{E}[T^{n}]$ exists iff 
	${\omega}_{r,n-r}$ (defined in Proposition \ref{prop-exp})
	exists, with $r=0,\ldots,n$. 
	By Jensen's inequality (see, e.g.,  \cite{clch:01}) we obtain
	\[
	|{\omega}_{0,1}|
	\leqslant
	\sqrt{\alpha^{2}\mathbb{E}[X^{2}]+4}
	=
	\sqrt{\alpha^{2}\left({2+3\delta^2\over 2+\delta^2}\right)+4}
	<+\infty,
	\]
	and by Minkowski inequality (see, e.g., \cite{nts:55}) we have
	\[
	|{\omega}_{1,1}|
	\leqslant 
	\sqrt{\mathbb{E}[X^2]}
	\sqrt{\alpha^{2}\mathbb{E}[X^{2}]+4}
	=
	\sqrt{{2+3\delta^2\over 2+\delta^2}}
	\sqrt{\alpha^{2}\left({2+3\delta^2\over 2+\delta^2}\right)+4}
	<+\infty.
	\]
	Then, the expected value $\mathbb{E}[T]$ and variance $\mathrm{Var}[T]$ always exist.
	Note also that higher order moments can also be easily obtained 
	from the expression of \, $\mathbb{E}[T^{n}]$.
\end{rem}
\begin{prop}\label{prop:02}
	Let $X\sim\text{ASN}(\delta)$ and\, $T =  a^{-\perp}(X)$.
	%where $a^{-\perp}(\cdot)$ denotes the inverse function of $a(\cdot)$.
	Then,  
	$T\sim\text{BBS}(\alpha,\beta,\delta)$.
\end{prop}
\begin{proof}
	Since $\mathbb{P}(a^{-\perp}(X)\leqslant t)=G(a(t))$,
	we have that the PDF of the RV $a^{-\perp}(X)$ is equal to
	$g(a(t))a'(t)$ \ $=f(t;\alpha,\beta,\delta)$.
\end{proof}
\begin{prop}\label{Property-relation-densities}
	Let $T\sim \chi^2_3$, where $\chi^2_3$ denotes the chi-squared distribution with 
	$3$ degrees of freedom. We have the following relation
	\begin{align*}
		f_{a^{-\perp}(\sqrt{T})}(t)\, 
		\left[(1/a(t)-\delta)^2+(1/a(t))^2\right]
		=
		2(2+\delta^2) \,
		f(t;\alpha,\beta,\delta), \quad t>0,
	\end{align*}
	where $f_{a^{-\perp}(\sqrt{T})}(\cdot)$ denotes the PDF of the RV
	$a^{-\perp}(\sqrt{T})$.
\end{prop}
\begin{proof}
	Since 
	$\mathbb{P}(a^{-\perp}(\sqrt{T})\leqslant t)=\mathbb{P}(T\leqslant a^2(t))$,
	we have that $f_{a^{-\perp}(\sqrt{T})}(t)=2a^2(t)\phi(a(t))a'(t)$, from where the proof follows.
\end{proof}
\subsection{Some properties of the HR of the BBS distribution}\label{properties-BBS}
Let
$
s(t)
= 
-{f'(t;\alpha,\beta, \delta)/ f(t;\alpha,\beta, \delta)}
$,
where $f(\cdot;\alpha,\beta, \delta)$ denotes the PDF of the BBS distribution \eqref{sec2:02}.
It is straightforward to show that
\[
s(t)=\frac{a'(t)m(t)}{(1-\delta a(t))^2+1},
\]
where
\[
m(t)
=
\delta^2 a^3(t)-2\delta a^2(t)-2a(t)(\delta^2-1)+2\delta-
\big((1-\delta a(t))^2+1\big)\frac{a''(t)}{[a'(t)]^2}.
\]
\begin{rem}\label{rem-us}
	%By \eqref{derivates} we have that $a'(t)=\big[(1/\alpha)\sqrt{t/\beta}-(1/2)a(t)\big]/t$ and
	%$a''(t)=-\big[(1/\alpha)\sqrt{t/\beta}-(3/4)a(t)\big]/t^2$.
	Using the
	identities in \eqref{derivates},
	note that $m(t)=0$ iff
	\begin{multline*}
		\frac{1}{\alpha^2\beta}\, t
		\left\{\delta^2 a^3(t)-2\delta a^2(t)-2(\delta^2-1)a(t)+2\delta\right\}
		+
		\frac{1}{4}a(t)
		\left\{\delta^2 a^4(t)-2\delta a^3(t)-(5\delta^2-2)a^2(t)+8\delta a(t)-6\right\}
		\\
		+
		\frac{1}{\alpha\sqrt{\beta}}\, \sqrt{t}
		\left\{-\delta^2 a^4(t)+2\delta a^3(t)+(3\delta^2-2)a^2(t)-4\delta a(t)+2\right\}
		=0.
	\end{multline*}
\end{rem}
Consider also the function spaces 
\[
B=\left\{
\ell:\mathbb{R}^+\to\mathbb{R}^{+} \ \mbox{differentiable}: \
\begin{array}{ccc}
\ell'(t)<0 \ \mbox{for} \ t\in(0,t_0), \ \ell'(t_0)=0, 
\\[0,1cm]
\ell'(t)>0 \ \mbox{for} \ t>t_0
\end{array}
\right\},
\]
\[
U=\left\{
\ell:\mathbb{R}^+\to\mathbb{R}^+ \ \mbox{differentiable}: \
\begin{array}{ccc}
\ell'(t)>0 \ \mbox{for} \ t\in(0,t_0), \ \ell'(t_0)=0, 
\\[0,1cm]
\ell'(t)<0 \ \mbox{for} \ t>t_0
\end{array}
\right\}.
\]
Each function $\ell\in B$ or $\ell \in U$ is said bathtub shaped or 
upside down bathtub shaped, respectively.

The following results due to \cite{Glaser} helps us
to characterize the shape of the failure rates, through the function $s(\cdot)$.
\begin{enumerate}
	\item 
	If $t\mapsto s(t)$ is increasing, then, the \textrm{HR} is increasing in $t$.
	\item 
	If $t\mapsto s(t)$ is decreasing, then,
	the \textrm{HR} is decreasing in $t$.
	\item 
	If $t\mapsto s(t)\in B$ and if there exist a $t^*$ such that $h'(t^*;\alpha,\beta,\delta)=0$, then,
	the \textrm{HR} belongs to
	$B$, otherwise the \textrm{HR} is increasing in $t$.
	\item 
	If $t\mapsto s(t)\in U$ and if there exist a $t^*$ such that $h'(t^*;\alpha,\beta,\delta)=0$, then, 
	the \textrm{HR}  belongs to
	$U$, otherwise the \textrm{HR} is decreasing in $t$.
\end{enumerate}
Using the expressions of the derivatives of $a(\cdot)$ in \eqref{derivates}
we can get the monotonicity of the 
\textrm{HR} of the BBS distribution from the following equation
\[
s'(t)
=
\frac{m'(t)a'(t)+m(t)a''(t) + 2\delta(1-\delta a(t))a'(t) s(t)}{(1-\delta a(t))^2+1},
\]
where 
\begin{align*}
	\frac{m'(t)}{a'(t)}
	=
	3\delta^2 a^2(t)
	-
	4\delta a(t)
	-
	2(\delta^2-1)
	+
	2\delta (1-\delta a(t))\, \frac{a''(t)}{[a'(t)]^2}
	-
	\big((1-\delta a(t))^2+1\big)\,
	\frac{a'''(t)a'(t)-2[a''(t)]^2}{[a'(t)]^4}.
\end{align*}

For example, if $\delta=0$ and $\alpha> 2$, we have that
$m(t)>0$ iff $t>\beta$.
Defining the set 
$
\mathscr{L}^\alpha_\beta
=
\left\{
t: 
t^4+ (4-\alpha^2)\beta t^3
\right.
$
$
\left.
+ 6(1-\alpha^2)\beta^2 t^2
+ (4+3\alpha^2)\beta^3 t + \beta^4<0
\right\}
$
note that
$
m'(t)
=
a'(t)
\big(
2
-
2\,
({a'''(t)a'(t)-2[a''(t)]^2)/[a'(t)]^4}
\big)
<0
$
on $\mathscr{L}^\alpha_\beta$. That is, 
$s'(t)=\left(m'(t)a'(t) + m(t)a''(t)\right)/2<0$ on
$\{t\in \mathscr{L}^\alpha_\beta:t>\beta\}$. Therefore, by Item 2 above, 
the HR $t\mapsto h(t;\alpha,\beta,\delta=0)$ is decreasing on 
$\{t\in \mathscr{L}^\alpha_\beta:t>\beta\}$.
On the other hand, if $\delta=0$ and $\alpha<1$,
$m(t)<0$ iff $t<\beta$. In this case, note that 
$m'(t)>0$ on $[\mathscr{L}^\alpha_\beta]^c=\mathbb{R}^+$, hence $s'(t)>0$ for each $t<\beta$.
Then, using the Item 1 above, the HR is increasing for each $t<\beta$.

Another easy case to study is when $\delta=1$. In this case, 
$m(t)>0$ iff $t>\beta$.
Note also that
$
m'(t)
<0
$
on the set $\mathscr{L}_{\alpha,\beta}=\{t:3a^2(t)-4a(t)+2(1-a(t))a''(t)/[a'(t)]^2<0\}$.
Then, $s'(t)<0$ on $\{t\in \mathscr{L}_{\alpha,\beta}:t>t_1\}$ where 
$t_1=\beta[\alpha+\sqrt{\alpha^2+4}]^2/4$. 
Therefore, by Item 2 above, 
the HR $t\mapsto h(t;\alpha,\beta,\delta=1)$ is decreasing on 
$\{t\in \mathscr{L}_{\alpha,\beta}:t>t_1\}$.
Similar analyzes can be done for the other possible cases.

We emphasize that $h'(t;\alpha,\beta,\delta)=0$ iff the PDF of the BBS distribution 
is a decreasing function. But, by Lemma \ref{lemma-mono} 
this happens when $\delta=0$ and $t>\beta$ or 
$t<t_0$ ( $t>t_0$ ), 
for each $\delta>0$ ( $\delta<0$ ) with 
$t_0=\beta [(\alpha/\delta)+\sqrt{(\alpha/\delta)^2+4}]^2/4$.
So, to see if the \textrm{HR} belongs (or not) to $B$ or to $U$ 
it would be sufficient to verify that $t\mapsto s(t)$ belongs (or not) to $B$ or to $U$.

\subsection{Bimodality Properties}
In this subsection, some results on the
bimodality properties of \text{BBS} distribution are obtained.
\begin{prop}
	A mode of the $\text{BBS}(\alpha,\beta,\delta)$ is any point 
	$t_0 =t_0(\alpha,\beta,\delta)$ that satisfies
	\[
	t_0=-\frac{\alpha^2\beta}{p_3(t)}
	\left[
	{1\over 4}a(t)p_4(t)+{1\over \alpha\sqrt{\beta}}\,\sqrt{t}\,\widetilde{p}_4(t) 
	\right],
	\]
	where $p_3(t)=\delta^2 a^3(t)-2\delta a^2(t)-2(\delta^2-1)a(t)+2\delta$, 
	$p_4(t)=\delta^2 a^4(t)-2\delta a^3(t)-(5\delta^2-2)a^2(t)+8\delta a(t)-6$ and
	$\widetilde{p}_4(t)=-\delta^2 a^4(t)+2\delta a^3(t)+(3\delta^2-2)a^2(t)-4\delta a(t)+2$.
\end{prop}
\begin{proof}
	A mode of the $\text{BBS}(\alpha,\beta,\delta)$ is any point 
	$t$ that satisfies $f'(t;\alpha,\beta, \delta)=0$. But this happens iff 
	$s(t)=0$ which is equivalent to $m(t)=0$, where $s(t)$ and $m(t)$ were defined in 
	Subsection \ref{properties-BBS}.
	Then, using Remark \ref{rem-us} and solving for $t$ gives the result.
\end{proof}
\begin{prop}\label{modes-diferents}
	The function $t\mapsto(g\circ a)(t)$ and the PDF of the \text{BBS} 
	distribution \eqref{sec2:02} have different modes.
\end{prop}
\begin{proof}
	We will do the proof by contradiction.
	Let's suppose that $t_0$ is a mode for both $(g\circ a)(\cdot)$ 
	(which always exists, since $g$ is bimodal) and $f(\cdot;\alpha,\beta,\delta)$.
	Then $g'(a(t_0))a'(t_0)=0$ and $g''(a(t_0))[a'(t_0)]^2<0$. 
	
	Since  $f'(t_0;\alpha,\beta,\delta)=0$ and $f''(t_0;\alpha,\beta,\delta)<0$, using    
	\eqref{1-derivate} and \eqref{2-derivate} we obtain that
	$g(a(t_0))=0$, which is impossible.	Then, the proof follows.
\end{proof}
\begin{rem}\label{rem-imp}
	As a consequence of the proof of the Proposition \ref{modes-diferents} we have that,
	if $t_0$ is a maximum point of $t\mapsto(g\circ a)(t)$ then, 
	the maximum points of the BBS distribution must be to the left side of $t_0$.
	On the other hand, if
	$t_1$ is a minimum point of $t\mapsto(g\circ a)(t)$ then, 
	the minimum points of the BBS distribution must be to the right side of $t_1$.
\end{rem}
\begin{prop}\label{unimodality-of-f}
	The PDF of the BBS distribution \eqref{sec2:02} has at most one mode when $\delta=0$.
\end{prop}
\begin{proof}
	If $\delta=0$, then the classical $\text{BS}(\alpha,\beta)$ distribution is obtained, i.e., 
	$f(t;\alpha,\beta,\delta)=f(t;\alpha,\beta)=\phi(a(t))a'(t)$, $t>0$. Differentiating $f(t;\alpha,\beta)$, 
	we obtain
	\[
	f'(t;\alpha,\beta)=\phi(a(t))\left(a''(t)-a(t)[a'(t)]^2\right).
	\]
	Using \eqref{derivates}, it is straightforward to show that $f'(t;\alpha,\beta)=0$ iff
	\begin{align}\label{eq-cubic}
		t^3+\beta(1+\alpha^2)t^2-\beta^2t-\beta^3=0.
	\end{align}
	The discriminant of a cubic polynomial $ax^3+bx^2+cx+d$ is given by 
	$\Delta_3=b^2c^2-4ac^3-4b^3d-27a^2d^2+18abcd$.
	In our case, we have
	\[
	\Delta_3=\beta^6\left(4(1+\alpha^2)^3+(1+\alpha^2)^2+18(1+\alpha^2)-23\right).
	\]
	Note that $\Delta_3>0$ for each $\alpha>0$, then the equation \eqref{eq-cubic} has three distinct real roots.
	
	Let $t_1,t_2$ and $t_3$ be the three distinct real roots of \eqref{eq-cubic}, 
	by Vieta's formula (see, e.g.,  \cite{vin:03}), it is valid that
	\begin{eqnarray*}
		t_1+t_2+t_3&=&-\beta(1+\alpha^2)
		\\
		t_1 \,t_2+t_1 \,t_3+t_2 \,t_3&=&-\beta^2
		\\
		t_1 \,t_2 \,t_3&=&\beta^3.
	\end{eqnarray*}
	From the first and third equations above we conclude that there must be two negative and one positive roots,
	hence $f(t;\alpha,\beta)$ has at most one mode.
\end{proof}
\begin{prop}\label{proposition-mode}
	If $\delta=-\alpha$, then one of the modes of the BBS distribution \eqref{sec2:02} occurs at $t=\beta.$
\end{prop}
\begin{proof}
	Since $a(\beta)=0$, $a'(\beta)={1}/{\alpha\beta}$, $a''(\beta)=-{1}/{\alpha\beta^{2}}$, 
	$g(0)={2}/{\sqrt{2\pi}(2+\delta^2)}$ and $g'(0)=-g(0)\delta$,
	by \eqref{1-derivate} we have that
	\begin{align}\label{ident-main}
		f'(\beta;\alpha,\beta,\delta)
		=
		-\frac{2}{\alpha\beta^{2}(2+\delta^{2})\sqrt{2\pi}}\left({\delta\over \alpha}+1\right)
		= 0
		\ \ \text{since} \ \
		\delta=-\alpha.
	\end{align}
	I.e., $t=\beta$ is one of the critical points of 
	$f$ when $\delta=-\alpha$.
	
	Since $a'''(\beta)={9}/{4\alpha\beta^{3}}$ and $g''(0)=-g'(0)(1+\delta^2)$,
	using \eqref{2-derivate}, note that
	\[
	f''(\beta;\alpha,\beta,\delta)
	=
	\frac{2}{\alpha\beta^{3}(2+\delta^{2}\sqrt{2\pi})}
	\left(\delta(1+\delta^2){1\over\alpha^2}+3{\delta\over\alpha}-{9\over 4} \right).
	\]

	As $\delta=-\alpha$ we obtain 
	\[
	f''(\beta;\alpha,\beta,\delta)
	=
	-\frac{2}{\alpha\beta^{3}(2+\alpha^{2})\sqrt{2\pi}}
	\left({1+\alpha^2\over\alpha}+{21\over 4}\right)<0.
	\]
	Therefore, the PDF of the BBS distribution is concave
	downward when  $\delta=-\alpha$.
\end{proof}
\begin{exmp}[Bimodality]
	Consider $\alpha=\beta=1$ and $\delta=-\alpha$. 
	By Proposition \ref{proposition-mode} the point $t=1$ is one of the modes of $f(\cdot;\alpha,\beta,\delta)$.
	Using \eqref{1-derivate} note that 
	$f'(t;\alpha,\beta,\delta)=0$ iff
	\[
	p(y)
	=
	y^{10}+2y^{9}+y^{6}-4y^5+3y^4-8y^3+4y^2+2y-1=0, \quad \text{where} \ y=t^{1/2}.
	\]
	We have that $p(0)=-1<0,$
	\[
	p(1/2)=\frac{85}{1024}>0,
	\quad 
	p(3/4)=-\frac{252223}{1048576}<0
	\ \ \text{and} \ \
	p(5/4)\approx 15.27127>0.
	\]
	Therefore, $p(y)$ has roots in the intervals $(0,1/2)$, $(1/2,3/4)$ and $(3/4,5/4)$. 
	It is not hard to show that $p(y)>0$ for $y>1$. Thus, all real roots of the
	polynomial $p(y)$ lie in the interval $(0,5/4)$. 
	Computationally it can be verified that $y_0\approx 0.419703$, $y_1\approx 0.646914$ and $y_2=1$ 
	are the only roots of $p(y)$ on $\{y:y>0\}$.
	Hence, $t_0=y_0^2\approx 0.1761$, $t_1=y_1^2\approx 0.4184$ and $t_2=y^2_2=1$ are the only roots of 
	$f'(t;\alpha,\beta,\delta)=0$.
	It can be verified that 
	\begin{table}[H]
		\centering  
		\begin{tabular}{|c|c|c|c|c|}
			\hline
			$a^{(k)}(t)$  &  {$k=0$}  & {$k=1$}   & $k=2$ & $k=3$ \\ \hline
			$t=t_0$ & $-1.9633$ & $7.963$ & $-61.0996$ & $848.6550$ \\ \hline
			$t=t_1$ & $-0.8991$  & $2.6204$ & $-7.5471$ & $42.8876$ \\ \hline
		\end{tabular}
		\quad \, \text{\resizebox{!}{0,2695cm}{and}}
		\\[0,28cm]
		\begin{tabular}{|c|c|c|c|}
			\hline
			$g^{(k)}(t)$  &  {$k=0$}  & {$k=1$}   & $k=2$  \\ \hline
			$t=t_0$ & $1.9219\,\phi(t_0)/{3}$ & $1.8585\,{\phi(t_0)}/{3}$ & $-0.0616\,{\phi(t_0)}/{3}$  \\ \hline
			$t=t_1$ & $1.0101\,\phi(t_1)/{3}$  & $1.1100\,{\phi(t_1)}/{3}$ & $2.1692\,{\phi(t_1)}/{3}$ \\ \hline
		\end{tabular}
	\end{table}
	\noindent
	where $a^{(0)}\equiv a$ and $g^{(0)}\equiv g$. Using \eqref{2-derivate} and the quantities above, we obtain
	\begin{align*}
		f''(t_0;\alpha,\beta,\delta)= g''(a(t_0))[a'(t_0)]^3+3g'(a(t_0))a'(t_0)a''(t_0)+g(a(t_0))a'''(t_0)
		\approx -1107.6637\,\frac{\phi(t_0)}{3}<0,
	\end{align*}
	and similarly
	$
	f''(t_1;\alpha,\beta,\delta)
	\approx 
	60.3992\, {\phi(t_1)}/{3}>0.
	$

	Therefore, the PDF of the BBS distribution, with parameters $\alpha=\beta=1$ and $\delta=-\alpha$,
	has exactly two modes at $t=t_0$ and $t=t_2$.
\end{exmp}
\begin{rem}
	Let $\alpha=\beta=1$ and $\delta=-\alpha$. 
	It can be verified that the point $t_{max}\approx a^{-\perp}(0.83929)=2.26240$ is the only maximum point of
	the function $(g\circ a)(\cdot)$. The Remark \ref{rem-imp} assures us that the 
	maximum points of the PDF $f(\cdot;\alpha,\beta,\delta)$ must be to the left side of $t_{max}$. 
	This statement was verified in the previous example.
\end{rem}
\subsection{Shannon Entropy}
For a continuous PDF $f(t)$ on an interval $I$, 
its entropy is defined as
\[
H(f) 
= 
-\int_I f(t)\log f(t) {\rm d}t.
\]
This definition of entropy, introduced by \cite{SW49}, 
resembles a formula for a thermodynamic notion of entropy.  
In our probabilistic context, if $X$ is an absolutely  continuous 
RV with PDF $f_X(t)$, the
quantity  $H(X)=H(f_X)=-\mathbb{E}[\log f_X(X)]$ is viewed as a measure of uncertainty associated with a 
RV. 
Note that $H(X)$ is not necessarily well-defined, since the integral does not always exist.

Consider $T\sim\text{BBS}(\alpha,\beta,\delta)$. The Shannon entropy of $T$ 
satisfies the following identity
\begin{prop} If $T\sim\text{BBS}(\alpha,\beta,\delta)$, there exists a constant $C(\alpha,\beta,\delta)$ such that
	\begin{align}\label{ident-entropy}
		H(T)
		=
		C(\alpha,\beta,\delta)
		+
		\mathbb{E}
		\left[ 
		\log\left\{T^{3/2}/(T+\beta)\over (1-\delta a(T))^2+1 \right\}
		\right].
	\end{align}
\end{prop}
\begin{proof}
	It is straightforward to verify  that
	\begin{align*}
		H(T)= \log(2+\delta^2)+\log(2\alpha\beta^{1/2})+\log(\sqrt{2\pi})
		+
		{1\over 2}\mathbb{E}[a^2(T)]
		+
		\int_{0}^{\infty}\log\left\{{t^{3/2}/(t+\beta)\over (1-\delta a(t))^2+1}\right\}f(t;\alpha,\beta,\delta){\rm d}t.
	\end{align*}
	Since $a(T)\sim\text{ASN}(\delta)$, by Proposition \ref{prop:01} we have
	%$
	%\mathbb{E}[a(T)]=-{2\delta\over 2+\delta^2}
	%$
	%and
	$
	\mathbb{E}[a^2(T)]=1 +{2\delta^2/(2+\delta^2)}.
	$
	Therefore, the identity \eqref{ident-entropy} is verified considering 
	$
	C(\alpha,\beta,\delta)
	=
	\log(2+\delta^2)+\log(2\alpha\beta^{1/2})+\log(\sqrt{2\pi})
	+
	(1 +{2\delta^2/(2+\delta^2)})/2.
	$
\end{proof}
\begin{rem}
	If $T\sim\text{BBS}(\alpha,\beta,\delta)$ and $T\geqslant 1$, then, the Shannon entropy always exists.
	In fact, by Jensen's inequality (see, e.g.,  \cite{clch:01}),
	Minkosky inequality (see, e.g., \cite{nts:55}) and Remark \ref{exist-vl} we obtain
	\begin{align*}
		&\left|\mathbb{E}[\log(T^{3/2})]\right|
		\leqslant
		\log\mathbb{E}[T^{3/2}]
		\leqslant
		\log(\mathbb{E}[T^{2}])^{1/2} + \log(\mathbb{E}[T])^{1/2}<+\infty,
		\\[0,2cm]
		&
		\left|\mathbb{E}[\log(T+\beta)]\right|
		\leqslant 
		\log(\mathbb{E}[T]+\beta)<+\infty \quad \text{and}
		\\[0,2cm]
		&\left|\mathbb{E}[\log\left((1-\delta a(T)^2+1\right)]\right|
		\leqslant
		\log\left(
		2+\delta^2\mathbb{E}[a^2(T)]-2\delta\mathbb{E}[a(T)]\right)<+\infty,
	\end{align*}
	because $a(T)\sim\text{ASN}(\delta)$.
	Then, using \eqref{ident-entropy} and the above inequalities, the proof follows.
\end{rem}
%
%
%
%
%

%% end color red %%%%%%%%%%%%%%%%%%%%%%%%%%%%%%%%%%%%%%%%%%%%%%%%%%%%%%%%%%%%%%%%%%%%%%%%%%%%%%%%%%%%%%%%%%%%%%%%%%%%%%%%%%%%%%%%5

\section{Estimation and inference}\label{sec:3}
\subsection{Maximum likelihood estimation}\label{sec:mleestim}

Let $(t_{1},\ldots,t_{n})$ be a random sample of size $n$ from the BBS
distribution with PDF in \eqref{sec2:02}. Considering $\delta$ known,
it follows that the log-likelihood function, without the constant, is given by
\begin{align*}
	\ell(\bm\theta) 
	= 
	-n\log(\alpha)-\frac{n}{2}\log(\beta)
	+  
	\sum_{i=1}^{n}  \log\left(1+(1- \delta a(t_i))^2\right)
	-\frac{1}{2} \sum_{i=1}^{n} a^2(t_i)
	+
	\sum_{i=1}^{n}  \log\left(t_{i}+\beta\right), 
\end{align*}
where ${\bm \theta} = (\alpha, \beta)$. 
Since
\begin{align}\label{first-order}
	{\partial \over \partial\alpha}a(t)
	=
	-{t^{-1/2}\over \alpha^2\beta^{1/2}}(t-\beta)
	\quad 
	\text{and}
	\quad 
	{\partial \over \partial\beta}a(t)
	=
	-{t^{-1/2}\over 2\alpha}(\beta^{-3/2}t+2),
\end{align}
taking the  first derivatives with respect to 
$\alpha$ and $\beta$ and equating them to zero, we have
\begin{align}\label{sec4.1:02}
	\frac{\partial}{\partial\alpha}\ell(\bm\theta) 
	&=
	-\frac{n}{\alpha}
	-
	\sum_{i=1}^{n}
	\left(
	{2\delta(1-\delta a(t_i))\over 1+(1-\delta a(t_i))^2}
	+
	a(t_i)
	\right)
	{\partial \over \partial\alpha}a(t_i)
	=
	0 
	\quad \text{and}
	\nonumber
	\\ 
	\\ 
	\frac{\partial}{\partial\beta} \ell(\bm\theta) 
	&=
	-\frac{n}{2\beta }
	-
	\sum_{i=1}^{n}
	\left(
	{2\delta(1-\delta a(t_i))\over 1+(1-\delta a(t_i))^2}
	+
	a(t_i)
	\right)
	{\partial \over \partial\beta}a(t_i)
	+
	\sum _{i=1}^n \frac{1}{t_i+\beta}
	=
	0. \nonumber
\end{align}
%

%{
%% begin color red %%%%%%%%%%%%%%%%%%%%%%%%%%%%%%%%%%%%%%%%%%%%%%%%%%%%%%%%%%%%%%%%%%%%%%%%%%%%%%%%%%%%%%%%%%%%%%%%%%%%%%%%%%%%%%%%
The ML estimates $\widehat{\alpha}$ and $\widehat{\beta}$ of $\alpha$ and $\beta$, 
respectively, are obtained by solving an iterative procedure for non-linear optimization 
of the system of equations in \eqref{sec4.1:02}, such as the Broyden-Fletcher-Goldfarb-Shanno (BFGS) quasi-Newton method; see \cite{mjm:00}. 
{
	% % The BFGS algorithm employs
	% % the same principle as the Newton-Raphson method, except that it uses a sequence of symmetric
	% % and positive definite matrices ${\bm C}^{(r)}$ instead of the Hessian matrix $-{\bm H}^{-1}$,
	% % such that $\lim_{r\rightarrow\infty}{\bm C}^{(r)}=-{\bm H}^{-1}$. An identity matrix of
	% % same order ${\bm C}^{(0)}$ is taken as the initial matrix, because it is symmetric and positive definite, thus leading to symmetric and positive definite approximations of ${\bm C}^{(r)}$. The recursive form of such matrices is expressed as
	% % $$
	% % {\bm C}^{(r+1)}={\bm C}^{(r)}-\frac{{\bm C}^{(r)}{\bm g}^{(r)}({\bm g}^{(r)})^{\top}{\bm C}^{(r)}}{({\bm g}^{(r)})^{\top}{\bm C}^{(r)}{\bm g}^{(r)}} +
	% %      \frac{{\bm h}^{(r)}({\bm h}^{(r)})^{\top}}{({\bm h}^{(r)})^{\top}{\bm g}^{(r)}}, \quad r=0,1,\ldots,
	% % $$
	% % where ${\bm g}^{(r)}={\bm\theta}^{(r+1)}-{\bm\theta}^{(r)}$ and ${\bm h}^{(r)}={\bm U}({\bm\theta}^{(r+1)})-{\bm U}({\bm\theta}^{(r)})$, with ${\bm U}$ being the score function. Thus, the maximum of $\ell({\bm \theta})$ is attained at
	% % $$
	% % {\bm\theta}^{(r+1)}={\bm\theta}^{(r)}-\lambda^{(r)}{\bm C}^{(r)}{\bm U}({\bm\theta}^{(r)}), \quad r=0,1,\ldots,
	% % $$
	% % where $\lambda^{(r)}$ is a scalar determined by a linear search procedure of
	% % ${\bm\theta}^{(r)}$ in the direction $-{\bm C}^{(r)}{\bm U}(\bm{\theta}^{(r)})$.
	We implement the BFGS algorithm in the \texttt{R} software, available at \url{http://cran.r-project.org}, by the function \texttt{optim}.
	
	The ML estimator $\widehat{\bm \theta}$, under some standard regularity conditions (see Subsection \ref{Confidence Intervals}), is consistent and follows a normal joint asymptotic distribution with mean ${\bm \theta}$
	and covariance matrix $\bm{\Sigma}(\widehat{\bm \theta})$. Furthermore, $\bm{\Sigma}(\widehat{\bm \theta})$ can be obtained from the corresponding expected Fisher information matrix, ${\cal I}({\bm \theta})$ say. Thus, we have
	$$
	\sqrt{n} \big(\widehat{{\bm \theta}} -{\bm \theta}\big) \quad \stackrel{D}{\rightarrow}\quad \textrm{N}_{2}\big(\bm{0}_{(2)\times 1},
	\bm{\Sigma}(\widehat{\bm \theta}) = {\cal J}({\bm \theta})^{-1}\big),\quad \text{as} \quad n \to \infty,
	$$
	where $\stackrel{D}{\rightarrow}$ denotes convergence in distribution, $\bm{0}_{(2)\times 1}$
	is a $(2)\times 1$ vector of zeros and
	$
	{\cal J}({\bm \theta}) = \lim_{n\to\infty}\frac{1}{n} {\cal I}({\bm \theta}).
	$
	Here, we approximate the expected Fisher information matrix by its observed version, and the square root of each diagonal
	element of its inverse matrix is used to approximate the associated
	standard error (SE); see \cite{eh:78}.

	We can use the profile log-likelihood for finding the value of $\delta$. In fact, this parameter is assumed to be fixed in the log-likelihood function, 
	because some difficulties in calculating it by the ML method were reported.} Generally,
two steps are required to estimate $\delta$: 
\begin{itemize}
	\item[i)] Let $\delta_{i}=i$ and for each $i=-20,\ldots,0,\ldots,20$ compute the ML 
	estimates of $\alpha$ and $\beta$ by solving the 
	system of equations in \eqref{sec4.1:02};
	\item[ii)] Select the final estimate of $\delta$ as the one which maximizes 
	the log-likelihood function and also select the associated estimates 
	of $\alpha$ and $\beta$ as final ones.
\end{itemize}

%% end color red %%%%%%%%%%%%%%%%%%%%%%%%%%%%%%%%%%%%%%%%%%%%%%%%%%%%%%%%%%%%%%%%%%%%%%%%%%%%%%%%%%%%%%%%%%%%%%%%%%%%%%%%%%%%%%%%

%
%
%
%
{Case of random censoring.}\label{rendomestima}
Suppose that the time to the event of interest is not 
completely observed and it may be subject to right censoring. 
Let $c_{i}$ denote the censoring time and $t_{i}$ the time to 
the event of interest. We observe $y_{i} = \min\{t_{i},c_{i} \}$, 
whereas $\tau_{i} = I(t_{i} \leq c_{i})$ is such that 
$\tau_{i} = 1$ if $y_{i}$ is the time to the event of 
interest and $\tau_{i} = 0$ if it is right censored, 
for $i = 1,\ldots, n$. Let ${\bm \theta}=(\alpha,\beta)$ 
denote the parameter vector of the BBS model 
given in \eqref{sec2:02} with $\delta$ known. From $n$ pairs of times and censoring 
indicators $(t_{1}, \tau_{1}),\ldots, (t_{n}, \tau_{n})$, 
the corresponding likelihood function obtained 
under uninformative censoring can be expressed as
%
%
%\begin{multline}\label{rbbs}
%L(\bm\theta) 
%= 
%\prod^{n}_{i=1}
%f(t_{i};\alpha,\beta,\delta)^{\tau_{i}}(1-F(t_{i};\alpha,\beta,\delta))^{1-\tau_{i}} 
%\\
%= 
%\prod^{n}_{i=1}\left\{\frac{1+(1-\delta{a(t_{i})})^2}{2+\delta^2}\phi(a(t_{i}))
%\frac{t_{i}^{-3/2}(t_{i}+\beta)}{2\alpha\,\beta^{1/2}}\right\}^{\tau_{i}} 
%\\
%\times \left[1 - \Phi(a(t_{i}))+
%\delta\left(\frac{2-\delta{t_{i}}}{2+\alpha^2}\right)\phi(a(t_{i}))\right]^{1-\tau_{i}}.
%\end{multline}
%
\begin{align}\label{rbbs}
	L(\bm\theta) 
	=
	\prod^{n}_{i=1}
	f(t_{i};\alpha,\beta,\delta)^{\tau_{i}}(1-F(t_{i};\alpha,\beta,\delta))^{1-\tau_{i}} 
	&= 
	\prod^{n}_{i=1}\left\{\frac{1+(1-\delta{a(t_{i})})^2}{2+\delta^2}\phi(a(t_{i}))
	\frac{t_{i}^{-3/2}(t_{i}+\beta)}{2\alpha\,\beta^{1/2}}\right\}^{\tau_{i}}  \nonumber
	\\[0,2cm] &\quad
	\times \left[1 - \Phi(a(t_{i}))+
	\delta\left(\frac{2-\delta{t_{i}}}{2+\alpha^2}\right)\phi(a(t_{i}))\right]^{1-\tau_{i}}.
\end{align}
Therefore, the log-likelihood function for the BBS 
model obtained from \eqref{rbbs} is given by
\begin{align}\label{rbbs1}
	\textstyle \ell(\bm\theta)  =
	-\omega_{i}\eta({\bm \theta}) + 
	\sum^{n}_{i=1}\tau_{i}\log(1+(1-\delta{a(t_{i})})^2) 
	&+ 
	\sum^{n}_{i=1} \tau_{i}\log(\phi(a(t_{i})))\nonumber  
	-\frac{3}{2}\sum^{n}_{i=1}\tau_{i}\log(t_{i})  + \sum_{i=1}^{n}\tau_{i}\log(t_{i}+\beta) 
	\nonumber\\[0,2cm]
	&+ \sum^{n}_{i=1}(1-\tau_{i})
	\log\left[
	1 - \Phi(a(t_{i}))+
	\delta\left(\frac{2-\delta{t_{i}}}{2+\alpha^2}\right)\phi(a(t_{i}))
	\right],
\end{align}
where 
$\omega_{i}= \sum^{n}_{i=1}\tau_{i}$ 
and 
$\eta({\bm \theta})=\log(2\alpha\beta^{1/2}(2+\delta^{2}))$. The parameter vector ${\bm \theta}$ may be estimated 
using an iterative procedure for non-linear optimization (BFGS method) of the log-likelihood function \eqref{rbbs1}. The estimation of 
$\delta$ can be performed using the profile log-likelihood as mentioned earlier in Section \ref{sec:mleestim}.

% % % %{
% % % %% begin color red %%%%%%%%%%%%%%%%%%%%%%%%%%%%%%%%%%%%%%%%%%%%%%%%%%%%%%%%%%%%%%%%%%%%%%%%%%%%%%%%%%%%%%%%%%%%%%%%%%%%%%%%%%%%%%%%
% % % \begin{rem} The moments estimators based on BBS distribution can be obtained using the moments equations below
% % % \begin{equation}\label{eqmom}
% % % g(\alpha,\beta,\delta) = \left\{
% % % \begin{array}{c}
% % % \mathbb{E}[T]=\frac{1}{n} \sum\limits_{i=1}^{n} t_i \\
% % % \mathbb{E}[T^2]=\frac{1}{n} \sum\limits_{i=1}^{n} t_i^2 \\
% % % \mathbb{E}[\frac{1}{T}]=\frac{n}{ \sum\limits_{i=1}^{n} \frac{1}{t_i}} 
% % % \end{array}
% % % \right.
% % % \quad 
% % % = g(t_1,t_2,\ldots, t_n),
% % % \end{equation}
% % % \end{rem}
% % % \noindent
% % % where $\mathbb{E}[g(T_1,T_2,\ldots, T_n)]  - g(\alpha,\beta,\delta) \equiv 0$ and $(T_1,T_2,\ldots,T_n)$ is a random sample from a BBS population.  However, 
% % % equations in \eqref{eqmom} do not have analytical solutions, and it is necessary to use some numerical method. When these estimators present analytical solutions, they are usually used as initial values in the iterative processes used to obtain ML estimators. In this work, as initial values, we fixed $\delta=0$ and we used the modified moment estimators proposed by \cite{ng:03} for the BS distribution.

%}
%% begin color red %%%%%%%%%%%%%%%%%%%%%%%%%%%%%%%%%%%%%%%%%%%%%%%%%%%%%%%%%%%%%%%%%%%%%%%%%%%%%%%%%%%%%%%%%%%%%%%%%%%%%%%%%%%%%%%%

\subsection{Confidence intervals}\label{Confidence Intervals}
%{
%% begin color red %%%%%%%%%%%%%%%%%%%%%%%%%%%%%%%%%%%%%%%%%%%%%%%%%%%%%%%%%%%%%%%%%%%%%%%%%%%%%%%%%%%%%%%%%%%%%%%%%%%%%%%%%%%%%%%%

In this subsection we present confidence intervals (CIs) for 
$S(t;\alpha,\beta,\delta)$, $\mathbb{E}[T]$ and $\mathrm{Var}[T]$, where 
$T\sim\mbox{BBS}(\alpha,\beta,\delta)$ and $\delta$ is known.

\sloppy  Let $\{T_{n},n \geq 1\}$ be a sequence of RVs. We will say that $\{T_n\}$ 
is asymptotically normal (AN) with mean $\mu_n$ and variance $\sigma_n^2$,
and write $T_n\sim AN(\mu_n,\sigma_n^2)$, if $\sigma_n>0$ and as $n\to\infty,$
\[
{T_n-\mu_n\over \sigma_n}\longrightarrow N(0,1).
\]
Here $\mu_n$ is not necessarily the mean of $T_n$ and $\sigma_n^2$, not
necessarily its variance. This is, for sufficiently large $n$, for each $t\in\mathbb{R}$ we can approximate the probability
$\mathbb{P}(T_n\leqslant t)$ by $\mathbb{P}(Z\leqslant ((t-\mu_n)/\sigma_n))$ where
$Z$ is $N(0,1)$.

Let 
${\bm \theta} = (\alpha, \beta)^{\top}$
in 
${\bm \Theta}$
and $\rho \in (0,1)$.
The random interval $(\underline{\theta}(T_1,\ldots,T_n), \overline{\theta}(T_1,\ldots,T_n))$ 
will be called a CI at confidence level $1-\rho$ for the parameter
${\bm\theta}$, provided that
\[
\mathbb{P}\left(\underline{\theta}(T_1,\ldots,T_n)
<
\bm\theta
<
\overline{\theta}(T_1,\ldots,T_n)\right)
\geqslant
1-\rho.
\]

In what follows, we assume $\ell(\bm\theta)$ holds the following standard regularity conditions:
\begin{enumerate}
	\item The parameter space, defined by ${\bm \Theta}$, is open and $\ell({\bm\theta})$ has a global maximum 
	at ${\bm \Theta}$; 
	\item For almost all $t$, the fourth-order log-likelihood derivatives with respect to 
	the model parameters exist and are continuous in an open subset of ${\bm\Theta}$ that contains 
	the true parameter ${\bm\theta}$; 
	\item The support set of $t\mapsto f(t;{\bm \theta},\delta)$, for ${\bm\theta}$ 
	in ${\bm \Theta}$, does
	not depend on ${\bm \theta}$; 
	%The support of $T$ does not depend on unknown parameters;
	\item The expected information matrix ${\cal I}({\bm\theta})$ is positive definite and finite.
	We remember that 
	the information matrix ${\cal I}({\bm \theta})$ is a $2\times 2$ matrix  
	with elements  ${\cal I}_{j,k}(\bm\theta)$ $j,k=1,2,$ defined by
	\[
	{\cal I}_{j,k}(\bm\theta)
	=
	\mbox{Cov}
	\left({\partial \over \partial\theta_j} \log f(T;{\bm \theta},\delta),   
	{\partial \over \partial\theta_k} \log f(T;{\bm \theta},\delta)\right) ,
	\quad 
	\theta_j,\theta_k\in \{\alpha,\beta\}.
	\]
\end{enumerate}  
These regularity conditions are not restrictive and hold for the models cited in this work.
{
	Let
	\begin{align*}
		{\cal V}(t;{\bm \theta},\delta)
		=
		\dfrac{2\delta[1+(1-\delta a(t))^2] - 4\delta^2(1-\delta a(t))^2}
		{[1+(1-\delta a(t))^2]^2}
		-1,
\quad 
		{\cal W}(t;{\bm \theta},\delta)
		=
		\dfrac{2\delta(1-\delta a(t))}
		{1+(1-\delta a(t))^2}
		+
		a(t).
	\end{align*}	
	The Fisher information matrix may also be written as
	\[
	{\cal I}_{j,k}(\bm\theta)
	=
	-
	\mathbb{E}
	\left(\dfrac{\partial^2}{\partial \theta_j \partial \theta_k}\log f(T;{\bm \theta},\delta)\right),
	\quad 
	\theta_j,\theta_k\in \{\alpha,\beta\},
	\]
	where
	\begin{align*}
		\dfrac{\partial^2}{\partial \alpha^2}\log f(t;{\bm \theta},\delta)
		&=
		\dfrac{1}{\alpha^2}
		+
		{\cal V}(t;{\bm \theta},\delta)
		\Big(\dfrac{\partial}{\partial \alpha}a(t)\Big)^2
		-
		{\cal W}(t;{\bm \theta},\delta)\dfrac{\partial^2}{\partial \alpha^2}a(t),
		\\[0,2cm]
		\dfrac{\partial^2}{\partial \beta^2}\log f(t;{\bm \theta},\delta)
		&=\!
		\dfrac{1}{4\beta^2}
		\!+\!
		{\cal V}(t;{\bm \theta},\delta)
		\Big(\dfrac{\partial}{\partial \beta}a(t)\Big)^2
		\!\!\!-\!
		{\cal W}(t;{\bm \theta},\delta)\dfrac{\partial^2}{\partial \beta^2}a(t)
		\!-\!
		\dfrac{1}{(t+\beta)^2},
		\\[0,2cm]
		\dfrac{\partial^2}{\partial \alpha\partial\beta}\log f(t;{\bm \theta},\delta)
		&=
		\dfrac{\partial^2}{\partial \beta\partial \alpha}\log f(t;{\bm \theta},\delta)
		=
		{\cal V}(t;{\bm \theta},\delta)
		\dfrac{\partial}{\partial \alpha}a(t)
		\dfrac{\partial}{\partial \beta}a(t)
		-
		{\cal W}(t;{\bm \theta},\delta)\dfrac{\partial^2}{\partial \alpha\partial\beta}a(t),
	\end{align*}
	The above
	first-order partial derivatives of $a(\cdot)$ with respect to 
	$\alpha$ and $\beta$
	were calculated in \eqref{first-order} and the  respective
	second-order partial derivatives are given by
	\begin{align*}
		\dfrac{\partial^2}{\partial \alpha^2}a(t)
		=
		\dfrac{2t^{-1/2}}{\alpha^3\beta^{1/2}}(t-\beta),
		\quad 
		\dfrac{\partial^2}{\partial \beta^2}a(t)
		=
		\dfrac{3t^{-1/2}\beta^{-5/2}}{4\alpha},
		\quad
		\dfrac{\partial^2}{\partial \alpha\partial\beta}a(t)
		=
		\dfrac{\partial^2}{\partial \beta\partial\alpha}a(t)
		=
		\dfrac{t^{-1/2}}{2\alpha^2}(\beta^{-3/2}+2).
	\end{align*}
	Here, the mixed partial differentiations are commutative
	at a given point ${\bm \theta}$ in $\mathbb{R}^2$
	because the corresponding functions have
	continuous second partial derivatives at that point
	(Schwarz's Theorem).
}
\subsubsection{Confidence Interval for \(S(t;{\bm{\theta}},\delta)\).}\label{sobre}

Let $\widehat{\alpha}$ and $\widehat{\beta}$ be
ML estimates of $\alpha$ and $\beta$, respectively. 
It is known that the ML estimate of
${\widehat{\bm\theta}}=(\widehat{\alpha},\widehat{\beta})^{\top}$ 
has normal 
asymptotic distribution, 
with null mean vector and asymptotic covariance matrix given by the inverse 
of the information matrix $I({\bm\theta})$. I.e.,
\[
\sqrt{n}({\widehat{\bm\theta}}-{\bm\theta}) \sim AN \left({\bm 0},\left[I(\bm\theta)\right]^{-1}\right).
\]

Since the function ${\bm\theta}\mapsto S(t;{\bm\theta},\delta)$, $\forall t>0$,
is continuously differentiable, by the Delta method we have
\begin{align*}
	\sqrt{n}
	\big(S(t;\widehat{\bm\theta},\delta)-S(t;{\bm\theta},\delta)\big)
	\sim 
	AN\left({\bm 0},J_{S}(\bm\theta)
	\,
	\left[I(\bm\theta)\right]^{-1}
	\,
	J_{S}(\bm\theta)^{\top}\right),
\end{align*}
where 
$
J_{S}(\bm\theta)
=
J_{S}({\bm\theta};t)
=
\begin{bmatrix}
{\partial \over \partial \alpha} S(t;{\bm\theta},\delta) & \
{\partial \over \partial \beta} S(t;{\bm\theta},\delta) 
\end{bmatrix}_{1\times 2}
$ 
is the Jacobian of the function ${\bm\theta}\mapsto S(t;{\bm\theta},\delta)$.

As $\widehat{\bm\theta}$ is a ML estimate of $\bm\theta$, 
the asymptotic variance of $S(t;\widehat{\bm\theta},\delta)$ can be estimated by
\begin{align*}
	\mbox{Var}\big[S(t;\widehat{\bm\theta},\delta)\big]
	&\approx
	J_{S}(\widehat{\bm\theta})
	\,
	\left[I(\widehat{\bm\theta})\right]^{-1}
	\,
	J_{S}(\widehat{\bm\theta})^{\top}.
\end{align*}
As $\widehat{\bm\theta}$ is consistent (because it is a ML estimate), 
by Slutsky's theorem we have
\begin{align}\label{cov1}
	\sqrt{n}
	\big(S(t;\widehat{\bm\theta},\delta)-S(t;{\bm\theta},\delta)\big)
	\sim 
	AN\big({\bm 0},\mbox{Var}\big[S(t;\widehat{\bm\theta},\delta)\big]\big).
\end{align}
If $0<\rho<1$, using \eqref{cov1}, the CI at confidence level $1-\rho$  for 
$S(t;{\bm\theta},\delta)$ is obtained from the following identity: 
\begin{align*}
	\lim_{n\to\infty}
	\mathbb{P}
	\left(
	\frac
	{
		|S(t;\widehat{\bm\theta},\delta)-S(t;{\bm\theta},\delta)|
	}
	{ 
		\widehat{\sigma}(t)
	}
	< 
	{z_{\rho/2}\over \sqrt{n}}
	\right)
	=
	\mathbb{P}(|Z|<z_{\rho/2})
	\geqslant
	1-\rho,
\end{align*}
where $z_{\rho/2}$ is the $\rho/2$-quantile of the normal distribution 
and $\widehat{\sigma}^2(t)=\mbox{Var}[S(t;\widehat{\bm\theta},\delta)]$.
Then, the random interval 
\begin{align}\label{pre-int-var}
	\textstyle
	\left(
	S(t;\widehat{\bm\theta},\delta)
	-
	{z_{\rho/2}\over \sqrt{n}} \widehat{\sigma}(t)
	\, , \,
	S(t;\widehat{\bm\theta},\delta)
	+
	{z_{\rho/2}\over \sqrt{n}} \widehat{\sigma}(t)
	\right)
\end{align}
is a CI at confidence level $1-\rho$ for 
$S(t;{\bm\theta},\delta)$, $\forall t>0$.

\subsubsection{Confidence Interval for
	\(\mathbb{E}[T|{\bm{\theta}}]= \mathbb{E}[T]\).}\label{esp}
Since $T\sim\mbox{BBS}({\bm{\theta}},\delta)$ is a positive RV, we have
the identity 
$
\mathbb{E}[T|{\bm{\theta}}]
=
\int_{0}^{\infty}S(t;{\bm\theta},\delta){\rm d}t.
$
Using this identity and denoting $\widehat{\sigma}^2(t)=\mbox{Var}[S(t;\widehat{\bm\theta},\delta)]$ 
note that $\eqref{pre-int-var}$ implies that the set 
\[
\textstyle
\left\{
\mathbb{E}[T|\widehat{\bm{\theta}}]
-
{z_{\rho/2}\over \sqrt{n}} \int_{0}^{\infty} \widehat{\sigma}(t) {\rm d}t
<
\mathbb{E}[T|{\bm{\theta}}]
<
\mathbb{E}[T|\widehat{\bm{\theta}}]
+
{z_{\rho/2}\over \sqrt{n}} \int_{0}^{\infty} 
\widehat{\sigma}(t) {\rm d}t
\right\}
\]
contains the set
\[
\textstyle
\left\{
S(t;\widehat{\bm\theta},\delta)
-
{z_{\rho/2}\over \sqrt{n}} \widehat{\sigma}(t)
<
S(t;{\bm\theta},\delta)
<
S(t;\widehat{\bm\theta},\delta)
+
{z_{\rho/2}\over \sqrt{n}} \widehat{\sigma}(t)
\right\}.
\]
Therefore, the random interval 
\[
\textstyle
\left(
\mathbb{E}[T|\widehat{\bm{\theta}}]
-
{z_{\rho/2}\over \sqrt{n}} \int_{0}^{\infty} \widehat{\sigma}(t) {\rm d}t
\, , \,
\mathbb{E}[T|\widehat{\bm{\theta}}]
+
{z_{\rho/2}\over \sqrt{n}} \int_{0}^{\infty} \widehat{\sigma}(t) {\rm d}t
\right)
\]
provides us a CI at confidence level $1-\rho$ for $\mathbb{E}[T|{\bm{\theta}}]$.
If the lower limit of the CI is negative, we will replace it with zero.
\subsubsection{Confidence Interval for 
	\(\mathrm{Var}[T|{\bm\theta}] = \mathrm{Var}[T]\).}

Let 
$
\widehat{L}_\pm(t)
=
S(t;\widehat{\bm\theta},\delta)
\pm
z_{\rho/2} \widehat{\sigma}(t)/\sqrt{n}
$
where $\widehat{\sigma}^2(t)=\mbox{Var}[S(t;\widehat{\bm\theta},\delta)]$,
$t>0$.
Assume that
$
\widehat{L}_-(t)>0,
$
otherwise we replace this lower limit with zero.

Let
$
\textstyle A=
\{
\widehat{L}_-(t)
<S(t;\bm\theta,\delta)<
\widehat{L}_+(t)
\}
$
and 
$
\textstyle
B=
\{
\int_{0}^{\infty} \widehat{L}_-(t) {\rm d}t
<\mathbb{E}[T|{\bm{\theta}}]<
\int_{0}^{\infty} \widehat{L}_+(t) {\rm d}t
\}.
$
Using the identity
$
\mathbb{E}[T^2|{\bm{\theta}}]
=
2\int_{0}^{\infty}tS(t;{\bm\theta},\delta){\rm d}t,
$
let's denote also
$
\textstyle
C=
\{
2 \int_{0}^{\infty}t \widehat{L}_-(t) {\rm d}t
<
\mathbb{E}[T^2|{\bm{\theta}}]
<
2 \int_{0}^{\infty}t \widehat{L}_+(t) {\rm d}t
\}
$
and
\begin{align*}
	\textstyle
	D=
	\left\{
	-(
	\int_{0}^{\infty} \widehat{L}_+(t) {\rm d}t
	)^2
	<-(\mathbb{E}[T|{\bm{\theta}}])^2<
	-(
	\int_{0}^{\infty} \widehat{L}_-(t) {\rm d}t
	)^2
	\right\}.
\end{align*}
Note that $A\subseteq B, C, D$ and $B\cap D=B$. Hence, if 
$
\big(
\widehat{L}_-(t)
,
\widehat{L}_+(t)
\big)
$ 
is a random CI for
$S(t;\bm\theta,\delta)$ with confidence coefficient $1-\rho$ (by Section \ref{sobre}), 
for each $t>0$, then 
$
\big(
\int_{0}^{\infty} \widehat{L}_-(t) {\rm d}t
,
\int_{0}^{\infty} \widehat{L}_+(t) {\rm d}t
\big)
$
and 
$
\big(
2 \int_{0}^{\infty}t \widehat{L}_-(t) {\rm d}t
,
2 \int_{0}^{\infty}t \widehat{L}_+(t) {\rm d}t
\big)
$
are also (random) CIs for $\mathbb{E}[T|{\bm{\theta}}]$ and
$\mathbb{E}[T^2|{\bm{\theta}}]$ respectively, with confidence coefficient $1-\rho$ each.

Since
\begin{align*}
	\textstyle
	\mathrm{I}_{\mathrm{Var}}=
	\Big\{
	2 J(\widehat{L}_-,\widehat{L}_+)
	<
	\mathrm{Var}(T|\bm\theta)
	<
	2 J(\widehat{L}_+,\widehat{L}_-)
	\Big\}
	\supseteq
	B\cap C\cap D=B\cap C,
\end{align*}
where $J$ denotes the operator $J(f,g)=\int_{0}^{\infty}t f(t) {\rm d}t
-
(
\int_{0}^{\infty} g(t) {\rm d}t
)^2$,
we have
\[
\mathbb{P}(\mathrm{I}_{\mathrm{Var}})
\geqslant
\mathbb{P}(B\cap C)\geqslant\mathbb{P}(B)+\mathbb{P}(C)-1\geqslant 1-2\rho.
\]
Therefore, 
$
\big(
2 J(\widehat{L}_-,\widehat{L}_+),
2 J(\widehat{L}_+,\widehat{L}_-)
\big)
$
is a (random) CI for
$\mathrm{Var}(T|\bm\theta)$ with confidence coefficient 
$1-2\rho$. Again, if the lower limit of 
the CI is negative, we will replace it with zero.
\begin{rem}
	Analogously to that done in Subsection \ref{sobre}, we can construct a CI for 
	the function
	$\log(-\log(S(t;\alpha,\beta,\delta)))$.
	%%Using this CI we can obtain finer CIs for 
	%%$\mathrm{Var}[T|{\bm\theta}]$ and $\mathbb{E}[T|{\bm\theta}]$,
	%%in the sense that the lower limits of said intervals will always be positive.
\end{rem}

%}
%% begin color red %%%%%%%%%%%%%%%%%%%%%%%%%%%%%%%%%%%%%%%%%%%%%%%%%%%%%%%%%%%%%%%%%%%%%%%%%%%%%%%%%%%%%%%%%%%%%%%%%%%%%%%%%%%%%%%%
\section{Monte Carlo simulation}\label{sec:4}
{
	Two MC simulation studies were carried out to evaluate the performance of the 
	ML estimators of the proposed BBS model. The first study considers simulated 
	data generated from the BBS distribution, whereas the second one has 
	as its data generating process the BS, log-normal (LN) and MXBS distributions. 
	All numerical evaluations were done in
	the \texttt{R} software; see \cite{r:18}. The used \texttt{R} codes are available upon request.
	
	\subsection{Simulation study 1}
	In this first study we evaluate the performance of the ML estimators for the proposed BBS model, considering the simulated data generated from the same model. The simulation scenario assumes the sample} 
sizes $n \in \{10, 50\}$, the values of the 
shape parameter as $\alpha \in \{0.10,0.50,1.00,1.50\}$, the values of the asymmetric parameter as $\delta \in \{-10,-5,-1,1,5,10\}$, and 10,000 MC replications. 
The censoring proportion is $p\in\{0.0,0.1,0.3\}$; see Section \ref{rendomestima}. 
Note that the values of the shape parameter $\alpha$ 
have been chosen in order to study the performance under low, moderate
and high skewness.

For each value of the parameter $\delta$, sample size and censoring proportion, 
the empirical values for the bias (Bias) and mean squared error (MSE) 
of the ML estimators are reported in Tables~\ref{tab:x}--\ref{tab:x1}. From these tables, 
note that, as the sample size increases, the ML estimators 
become more efficient, as expected. 
We can also note that, 
as the censoring proportion increases, 
the performances of the estimators of $\alpha$ and 
$\beta$, deteriorate. 
It is interesting to note two points on the increasing of the bias 
of $\widehat\beta$: (i) when the skewness increases, the bias of $\widehat\beta$ 
increases, which is expected as the original distribution occurs 
in the BS, see for example \cite{lsc:08}; and (ii) note that there seems to be an increase in the bias of 
$\widehat\beta$ when we decrease the values 
of the parameter $\delta$, see the cases $\delta = \{-1,1\}$.
In general, all of these results show the 
good performance of the proposed model.

\begin{table}[!ht]
	\footnotesize
	\centering
	\caption{Simulated values of biases (MSEs within parentheses) of the estimators of the BBS model.}
	\label{tab:x}
	\renewcommand{\arraystretch}{1.3}
	\resizebox{\linewidth}{!}{
		\begin{tabular}{clrrrrrrrrrrrrrrrrrr}
			\hline
			&        &       &&\multicolumn{2}{c}{BBS($\alpha=0.1,\beta=1.0,\delta$)} && \multicolumn{2}{c}{BBS($\alpha=0.5,\beta=1.0,\delta$)}    \\ \cline{5-6} \cline{8-9} \\[-0.5cm]
			censoring \%&  $n$          & \multicolumn{1}{c}{$\delta$}       &&\multicolumn{1}{c}{Bias($\widehat{\alpha}$)}&\multicolumn{1}{c}{Bias($\widehat{\beta}$)}&&
			\multicolumn{1}{c}{Bias($\widehat{\alpha}$)}&\multicolumn{1}{c}{Bias($\widehat{\beta}$)}\\ \hline \\ [-0.2cm]
			
			0\%&10       &$-$10 &&$-$0.0011\,(0.0002)   &$-$0.0010\,(0.0009)   &&$-$0.0066\,(0.0062) &   0.0028\,(0.0184)   \\[-0.2cm]
			&         &$-$5  &&$-$0.0021\,(0.0003)   &$-$0.0034\,(0.0013)   &&$-$0.0117\,(0.0075) &$-$0.0042\,(0.0224)   \\[-0.2cm]
			&         &$-$1  &&$-$0.0240\,(0.0012)   &$-$0.0418\,(0.0054)   &&$-$0.1224\,(0.0308) &$-$0.1581\,(0.0824)    \\[-0.2cm]
			&         &1     &&$-$0.0235\,(0.0012)   &   0.0480\,(0.0066)   &&$-$0.1192\,(0.0303) &   0.2824\,(0.2153)    \\[-0.2cm]
			&         &5     &&$-$0.0022\,(0.0003)   &   0.0050\,(0.0015)   &&$-$0.0111\,(0.0078) &   0.0293\,(0.0376)    \\[-0.2cm]
			&         &10    &&$-$0.0012\,(0.0002)   &   0.0020\,(0.0010)   &&$-$0.0056\,(0.0063) &   0.0169\,(0.0244)    \\[0.1cm]
			%%%%%%
			&50       &$-$10 &&$-$0.0001\,($<$0.0001)&$-$0.0001\,(0.0001)   &&$-$0.0012\,(0.0004) &   0.0004\,(0.0026)    \\[-0.2cm]
			&         &$-$5  &&$-$0.0003\,($<$0.0001)&   0.0002\,(0.0001)   &&$-$0.0025\,(0.0010) &   0.0027\,(0.0028)   \\[-0.2cm]
			&         &$-$1  &&$-$0.0024\,(0.0002)   &$-$0.0182\,(0.0018)   &&$-$0.0139\,(0.0040) &$-$0.0597\,(0.0254)    \\[-0.2cm]
			&         &1     &&$-$0.0024\,(0.0002)   &   0.0205\,(0.0022)   &&$-$0.0133\,(0.0040) &   0.0927\,(0.0475)   \\[-0.2cm]
			&         &5     &&$-$0.0003\,($<$0.0001)&   0.0001\,(0.0001)   &&$-$0.0019\,(0.0010) &   0.0019\,(0.0028)    \\[-0.2cm]
			&         &10    &&$-$0.0001\,($<$0.0001)&   0.0003\,(0.0001)   &&$-$0.0010\,(0.0009) &   0.0034\,(0.0026)   \\[-0.2cm]
			
			&         &      &&                      &                      &&                    &      \\[-0.2cm]
			
			10\%&10   &$-$10 &&0.0039\,(0.0009)      &$-$0.0241\,(0.0081)   &&$-$0.0084\,(0.0083) &0.0393\,(0.0392)      \\[-0.2cm]
			&         &$-$5  &&0.0070\,(0.0013)      &$-$0.0466\,(0.0133)   &&$-$0.0184\,(0.0112) &0.0276\,(0.0575)      \\[-0.2cm]
			&         &$-$1  &&$-$0.0006\,(0.0011)   &$-$0.1308\,(0.0183)   &&0.0023\,(0.0301)    &$-$0.4827\,(0.2416)      \\[-0.2cm]
			&         &1     &&0.0012\,(0.0011)      &0.1533\,(0.0258)      &&0.0431\,(0.0381)    &1.0269\,(1.1883)      \\[-0.2cm]
			&         &5     &&$-$0.0029\,(0.0004)   &0.0002\,(0.0043)      &&$-$0.0233\,(0.0095) &0.0376\,(0.1120)      \\[-0.2cm]
			&         &10    &&$-$0.0007\,(0.0004)   &$-$0.0047\,(0.0033)   &&$-$0.0059\,(0.0079) &0.0281\,(0.1520)      \\[0.1cm]
			%%%%%%
			&50       &$-$10 &&0.0019\,(0.0002)      &0.0109\,(0.0022)      &&0.0053\,(0.0011)    &0.0659\,(0.0074)      \\[-0.2cm]
			&         &$-$5  &&0.0070\,(0.0007)      &$-$0.0027\,(0.0062)   &&0.0095\,(0.0015)    &0.1094\,(0.0170)      \\[-0.2cm]
			&         &$-$1  &&0.0091\,(0.0003)      &$-$0.1387\,(0.0195)   &&0.0391\,(0.0083)    &$-$0.5053\,(0.2571)      \\[-0.2cm]
			&         &1     &&0.0116\,(0.0004)      &0.1632\,(0.0271)      &&0.0904\,(0.0171)    &1.0871\,(1.2133)      \\[-0.2cm]
			&         &5     &&0.0015\,(0.0003)      &$-$0.0118\,(0.0031)   &&0.0001\,(0.0011)    &$-$0.0680\,(0.0078)      \\[-0.2cm]
			&         &10    &&0.0006\,(0.0001)      &$-$0.0094\,(0.0008)   &&0.0003\,(0.0011)    &$-$0.0423\,(0.0040)      \\[-0.2cm]
			
			&         &      &&                      &                      &&                    &      \\[-0.2cm]
			
			30\%&10   &$-$10 &&0.0875\,(0.0123)      &$-$0.2085\,(0.0677)   &&0.0195\,(0.0232)    &0.0854\,(0.1273)     \\[-0.2cm]
			&         &$-$5  &&0.0811\,(0.0106)      &$-$0.2206\,(0.0638)   &&0.0035\,(0.0255)    &0.0741\,(0.1459)      \\[-0.2cm]
			&         &$-$1  &&0.0031\,(0.0013)      &$-$0.1311\,(0.0187)   &&$-$0.0680\,(0.0286) &$-$0.4564\,(0.2194)      \\[-0.2cm]
			&         &1     &&0.0004\,(0.0012)      &0.1497\,(0.0250)      &&0.0521\,(0.0550)    &1.0430\,(1.4937)      \\[-0.2cm]
			&         &5     &&0.0410\,(0.0062)      &$-$0.0509\,(0.0390)   &&0.0283\,(0.0343)    &0.1593\,(0.4103)      \\[-0.2cm]
			&         &10    &&0.0679\,(0.0102)      &$-$0.1303\,(0.0538)   &&0.0307\,(0.0290)    &0.0776\,(0.2059)      \\[0.1cm]
			%%%%%%
			&50       &$-$10 &&0.1319\,(0.0185)      &$-$0.2877\,(0.0945)   &&0.0324\,(0.0037)    &0.1056\,(0.0234)      \\[-0.2cm]
			&         &$-$5  &&0.1144\,(0.0139)      &$-$0.2838\,(0.0828)   &&0.0483\,(0.0082)    &0.1750\,(0.0597)      \\[-0.2cm]
			&         &$-$1  &&0.0126\,(0.0005)      &$-$0.1400\,(0.0199)   &&$-$0.0200\,(0.0056) &$-$0.4831\,(0.2352)      \\[-0.2cm]
			&         &1     &&0.0124\,(0.0005)      &0.1635\,(0.0273)      &&0.1065\,(0.0219)    &1.0849\,(1.2121)      \\[-0.2cm]
			&         &5     &&0.0657\,(0.0097)      &$-$0.0825\,(0.0538)   &&0.0101\,(0.0039)    &$-$0.0336\,(0.0613)      \\[-0.2cm]
			&         &10    &&0.0913\,(0.0145)      &$-$0.1506\,(0.0695)   &&0.0189\,(0.0066)    &$-$0.0005\,(0.0656)      \\[-0.1cm]
			
			\hline
		\end{tabular}
	}
\end{table}

\begin{table}[htb]
	\footnotesize
	\centering
	\caption{Simulated values of biases (MSEs within parentheses) of the estimators of the BBS model.}
	\label{tab:x1}
	\renewcommand{\arraystretch}{1.3}
	\resizebox{\linewidth}{!}{
		\begin{tabular}{clrrrrrrrrrrrrrrrrrr}
			\hline
			&         &       &&\multicolumn{2}{c}{BBS($\alpha=1.0,\beta=1.0,\delta$)} && \multicolumn{2}{c}{BBS($\alpha=1.5,\beta=1.0,\delta$)}    \\ \cline{5-6} \cline{8-9} \\[-0.5cm]
			censoring \%&$n$           & \multicolumn{1}{c}{$\delta$}       &&\multicolumn{1}{c}{Bias($\widehat{\alpha}$)}&\multicolumn{1}{c}{Bias($\widehat{\beta}$)}&&
			\multicolumn{1}{c}{Bias($\widehat{\alpha}$)}&\multicolumn{1}{c}{Bias($\widehat{\beta}$)}\\ \hline\\ [-0.2cm]
			
			0\% &10       &$-$10 &&$-$0.0162\,(0.0263)   &   0.0202\,(0.0507)   &&$-$0.0244\,(0.0608) &   0.0248\,(0.0728)   \\[-0.2cm]
			&         &$-$5  &&$-$0.0258\,(0.0326)   &   0.0070\,(0.0575)   &&$-$0.0477\,(0.0779) &   0.0166\,(0.0859)   \\[-0.2cm]
			&         &$-$1  &&$-$0.2455\,(0.1298)   &$-$0.2136\,(0.1767)   &&$-$0.3764\,(0.3075) &$-$0.2438\,(0.2449)    \\[-0.2cm]
			&         &1     &&$-$0.2441\,(0.1295)   &   0.5896\,(0.9632)   &&$-$0.3801\,(0.3087) &   0.8931\,(2.3962)    \\[-0.2cm]
			&         &5     &&$-$0.0269\,(0.0314)   &   0.0728\,(0.1911)   &&$-$0.0411\,(0.0714) &   0.0980\,(0.3867)    \\[-0.2cm]
			&         &10    &&$-$0.0160\,(0.0253)   &   0.0367\,(0.0831)   &&$-$0.0258\,(0.0588) &   0.0537\,(0.1826)    \\[0.1cm]
			%%%%%
			&50       &$-$10 &&$-$0.0023\,(0.0037)   &   0.0014\,(0.0073)   &&$-$0.0033\,(0.0083) &   0.0012\,(0.0105)    \\[-0.2cm]
			&         &$-$5  &&$-$0.0042\,(0.0042)   &   0.0046\,(0.0081)   &&$-$0.0049\,(0.0093) &   0.0058\,(0.0114)   \\[-0.2cm]
			&         &$-$1  &&$-$0.0391\,(0.0179)   &$-$0.0630\,(0.0486)   &&$-$0.0703\,(0.0456) &$-$0.0616\,(0.0618)    \\[-0.2cm]
			&         &1     &&$-$0.0385\,(0.0175)   &   0.1410\,(0.1272)   &&$-$0.0728\,(0.0448) &   0.1612\,(0.1832)   \\[-0.2cm]
			&         &5     &&$-$0.0035\,(0.0042)   &   0.0029\,(0.0078)   &&$-$0.0050\,(0.1832) &   0.0021\,(0.0115)    \\[-0.2cm]
			&         &10    &&$-$0.0033\,(0.0037)   &   0.0057\,(0.0072)   &&$-$0.0030\,(0.0115) &   0.0068\,(0.0107)   \\[-0.2cm]
			
			&         &      &&                      &                      &&                    &      \\[-0.2cm]
			
			10\%&10   &$-$10 &&0.0088\,(0.0575)      &0.1620\,(0.4960)      &&0.0456\,(0.1813)    &0.2316\,(0.6773)      \\[-0.2cm]
			&         &$-$5  &&$-$0.0074\,(0.0713)   &0.2010\,(0.2840)      &&0.0573\,(0.2973)    &0.4123\,(1.6875)      \\[-0.2cm]
			&         &$-$1  &&0.0005\,(0.1111)      &$-$0.6838\,(0.4792)   &&$-$0.0498\,(0.1986) &$-$0.7573\,(0.5941)      \\[-0.2cm]
			&         &1     &&0.1053\,(0.1851)      &2.4846\,(7.1727)      &&0.1370\,(0.3345)    &3.7838\,(18.1242)      \\[-0.2cm]
			&         &5     &&$-$0.0250\,(0.0393)   &0.0338\,(0.5171)      &&$-$0.0236\,(0.1127) &0.0605\,(1.5485)      \\[-0.2cm]
			&         &10    &&$-$0.0142\,(0.0276)   &0.0270\,(0.1993)      &&$-$0.0380\,(0.0587) &0.0121\,(0.2071)      \\[0.1cm]
			%%%%%
			&50       &$-$10 &&0.0187\,(0.0053)      &0.1260\,(0.0256)      &&0.0272\,(0.0115)    &0.1458\,(0.0361)      \\[-0.2cm]
			&         &$-$5  &&0.0219\,(0.0060)      &0.2016\,(0.0569)      &&0.0314\,(0.0134)    &0.2368\,(0.0820)      \\[-0.2cm]
			&         &$-$1  &&0.0746\,(0.0268)      &$-$0.7159\,(0.5141)   &&0.0963\,(0.0492)    &$-$0.7910\,(0.6272)      \\[-0.2cm]
			&         &1     &&0.2470\,(0.0974)      &2.5919\,(6.9146)      &&0.3733\,(0.2098)    &3.4797\,(12.8450)      \\[-0.2cm]
			&         &5     &&0.0028\,(0.0043)      &$-$0.1110\,(0.0205)   &&0.0004\,(0.0092)    &$-$0.1286\,(0.0275)      \\[-0.2cm]
			&         &10    &&0.0057\,(0.0043)      &$-$0.0619\,(0.0103)   &&0.0077\,(0.0093)    &$-$0.0763\,(0.0146)      \\[-0.2cm]
			
			&         &      &&                      &                      &&                    &      \\[-0.2cm]
			
			30\%&10   &$-$10 &&0.2393\,(0.3985)      &0.6891\,(4.0713)      &&0.4426\,(1.2441)    &1.0896\,(7.0831)      \\[-0.2cm]
			&         &$-$5  &&0.1985\,(0.2985)      &0.6169\,(1.9729)      &&0.7262\,(2.1753)    &1.9417\,(12.9445)      \\[-0.2cm]
			&         &$-$1  &&$-$0.1458\,(0.1069)   &$-$0.6624\,(0.4552)   &&$-$0.3442\,(0.2392) &$-$0.7627\,(0.5970)      \\[-0.2cm]
			&         &1     &&0.1667\,(0.3212)      &2.4961\,(7.9861)      &&0.2472\,(0.5322)    &3.5858\,(15.6473)      \\[-0.2cm]
			&         &5     &&0.0959\,(0.3339)      &0.4870\,(5.5723)      &&0.2082\,(1.0045)    &0.7895\,(11.7672)      \\[-0.2cm]
			&         &10    &&0.0774\,(0.1579)      &0.1944\,(0.9839)      &&0.1840\,(0.6911)    &0.4070\,(4.2938)      \\[0.1cm]
			%%%%%
			&50       &$-$10 &&0.1884\,(0.3646)      &0.6260\,(8.5780)      &&0.1899\,(0.1168)    &0.3775\,(0.4992)      \\[-0.2cm]
			&         &$-$5  &&0.2240\,(0.1207)      &0.6080\,(0.7797)      &&0.4197\,(0.4973)    &0.9224\,(2.5520)      \\[-0.2cm]
			&         &$-$1  &&$-$0.1144\,(0.0284)   &$-$0.6953\,(0.4855)   &&$-$0.2762\,(0.1004) &$-$0.7957\,(0.6349)      \\[-0.2cm]
			&         &1     &&0.3004\,(0.1388)      &2.5498\,(6.7053)      &&0.4939\,(0.3467)    &3.5056\,(12.8719)      \\[-0.2cm]
			&         &5     &&0.0436\,(0.0771)      &0.0425\,(1.5990)      &&0.0406\,(0.0558)    &$-$0.0484\,(0.9883)      \\[-0.2cm]
			&         &10    &&0.0347\,(0.0223)      &$-$0.0014\,(0.4995)   &&0.0410\,(0.0156)    &$-$0.0309\,(0.0156)      \\[-0.1cm]
			\hline
		\end{tabular}
	}
\end{table}

{
	\subsection{Simulation study 2}
	
	In this second simulation study we consider the BS, LN and MXBS distributions as data generating processes, and the BBS and BBSO distributions are fitted to the simulated data. 
	Note that the BBS and BBSO models are the closest competitors, since both models do not require mixture of distributions to produce bimodality. The purpose is to evaluate how the estimators behave when the data generating process is wrong (the assumed model is different from the data generating model). In addition, we also compare the adjustments of the BBS and BBSO models by means of the fitted log-likelihood (log-lik) values. The $\text{BS}(\alpha,\beta)$, $\text{LN}(\mu,\sigma)$ and $\text{MXBS}(\alpha_{1},\beta_{1},\alpha_{2},\beta_{2},p)$ samples were generated by considering the following PDFs 
	$f_{\text{BS}}(t;\alpha,\beta)=\phi(a(t)){[t^{-3/2}(t+\beta)]}/{[2\alpha\,\beta^{1/2}]}$, 
	$f_{\text{LN}}(t;\mu,\sigma)=1/[t\sigma\sqrt{2\pi}]\exp([\log(t)-\mu]^2/[2\sigma^{2}])$ and
	$f_{\text{MXBS}}(t;\alpha_1,\beta_1,\alpha_2,\beta_2,p)=pf_{\text{BS}}(t;\alpha_1,\beta_1)+[1-p]f_{\text{BS}}(t;\alpha_2,\beta_2)$, $t>0$, where $\phi(\cdot)$ and $a(\cdot)$ are as in \eqref{sec1:02} and \eqref{at}. Moreover, the BBSO PDF is given by 
	$f_{\text{BBSO}}(t;\alpha,\beta,\gamma)=[t^{-3/2}(t+\beta)]/[4\alpha\beta^{1/2}\Phi(-\gamma)]\phi(|a(t)|+\gamma)$, where $\phi(\cdot)$, $\Phi(\cdot)$ and $a(\cdot)$ are as in \eqref{sec1:02} and \eqref{at}.

	The simulation scenario considers: sample sizes $n \in \{10, 50\}$,
	the values of the shape parameters as $\alpha,\sigma \in \{0.10,1.00,1.50,2.50,4.00\}$,
	the values of the mixing parameter as $p \in \{0.25,0.50,0.75\}$, and 1,000 MC replications. 
	In this case,  we do not consider censoring as in Simulation 1. The values of the shape parameters $\alpha,\sigma$ cover different levels of skewness. Note that the $\text{BS}(\alpha,\beta)$ and $\text{LN}(\mu,\sigma)$ PDFs are unimodal, whereas the $\text{MXBS}(\alpha_{1},\beta_{1},\alpha_{2},\beta_{2},p)$ PDF is either unimodal or bimodal. In special, the parameters of the latter distribution have been chosen to provide bimodal shapes.

	The ML estimation results are presented in Tables~\ref{tab:BS-LN} and \ref{tab:MXBS}. The empirical means for the ML estimates and fitted log-likelihood values are reported.  A look at the results in Tables~\ref{tab:BS-LN} and \ref{tab:MXBS} allows us to conclude that the proposed BBS model provides better adjustment compared to the BBSO model based on the log-likelihood values.

	\begin{table}[!ht]
		\footnotesize
		\centering
		\caption{{Empirical mean from simulated BS and LN data for the indicated model, estimator, generator, $\alpha$ and $n$.}}
		\label{tab:BS-LN}
		\renewcommand{\arraystretch}{1.0}
		\resizebox{\linewidth}{!}{
			\begin{tabular}{lrrrrrrrrrrrrrrrrrrrrrr}
				\hline
				Generator &    &          && \multicolumn{4}{c}{BBS} &&  \multicolumn{4}{c}{BBSO} \\\cline{5-8}\cline{10-13}\\ [-0.2cm]
				
				& $n$  &   $\alpha$         &&   \multicolumn{1}{c}{$\widehat\alpha$} & \multicolumn{1}{c}{$\widehat\beta$} & \multicolumn{1}{c}{$\widehat\delta$} &  \multicolumn{1}{c}{log-lik} &&   \multicolumn{1}{c}{$\widehat\alpha$} & \multicolumn{1}{c}{$\widehat\beta$} & \multicolumn{1}{c}{$\widehat\gamma$} &  \multicolumn{1}{c}{log-lik}
				\\\hline\\ [-0.2cm]
				BS($\alpha$,$\beta=1.0$)     & 10 & 0.50 && 0.3442 &  1.0235 &  0.2160  & $-$4.4525  &&  0.2616 &  1.0101 & $-$1.6996 &$-$5.1787 \\
				&    & 1.00 && 0.6927 &  1.0639 &  0.0270  & $-$10.6575 &&  0.5264 &  1.0486 & $-$1.6618 &$-$11.3504\\ 
				&    & 1.50 && 1.0527 &  1.0904 & $-$0.0320& $-$13.8437 &&  0.7942 &  1.0798 & $-$1.6259 &$-$14.5051 \\
				&    & 2.50 && 1.7765 &  1.1310 & $-$0.0160& $-$17.1733 &&  1.3337 &  1.1228 & $-$1.5805 &$-$17.7920  \\
				&    & 4.00 && 2.8464 &  1.1624 & $-$0.0280& $-$19.4770 &&  2.1465 &  1.1526 & $-$1.5497 &$-$20.0728 \\[0.05cm]
				& 50 & 0.50 && 0.4750 &   1.0187 & $-$ 0.0020 &$-$33.6220 &&  0.3385 &  1.0027 & $-$1.2025 &$-$36.7761 \\
				&    & 1.00 && 0.9652 &   1.0376 &     0.0180 &$-$64.9110 &&  0.6777 &  1.0110 & $-$1.1975 &$-$67.9465 \\ 
				&    & 1.50 && 1.4547 &   1.0266 &     0.0030 &$-$81.0343 &&  1.0176 &  1.0170 & $-$1.1928 &$-$83.9877 \\
				&    & 2.50 && 2.4309 &   1.0236 &  $-$0.0130 &$-$98.0011 &&  1.6989 &  1.0212 & $-$1.1869 &$-$100.8794 \\
				&    & 4.00 && 3.8981 &   1.0188 &  $-$0.0240 &$-$109.9385&&  2.7212 &  1.0229 & $-$1.1836 &$-$112.7911 \\[0.1cm]
				LN($\mu=1.0$,$\sigma$) & 10 & 0.50 && 0.3570 & 1.0256 & 0.1610   & $-$4.6916  && 0.2720 & 1.0140& $-$1.6841& $-$5.4663 \\
				&    & 1.00 && 0.7970 &  1.1156 &  0.0710 & $-$11.4631 &&  0.6112 &  1.0936 & $-$1.5989& $-$12.3236\\
				&    & 1.50 && 1.3940  & 1.2790 &  0.0310 &$-$15.3910  && 1.1025  & 1.2483 & $-$1.4743& $-$16.3613\\
				&    & 2.50 && 3.6272  & 1.9682 & $-$0.1580 &$-$20.5330  && 3.0648  & 1.9584 & $-$1.2293& $-$21.7278\\
				&    & 4.00 && 14.4744 &  5.8215 &  0.0030& $-$26.1805 && 13.3911 &  6.1023 & $-$0.9624& $-$27.6768\\[0.05cm]
				& 50 & 0.50 && 0.4888  &  1.0238 &  0.0000 &$-$34.9625 &&  0.3564 &  1.0034 & $-$1.1701 &$-$38.7044 \\
				&    & 1.00 && 1.0946  & 1.0938  & 0.0070  &$-$69.8828 &&  0.8280 &  1.0206 & $-$1.0800 &$-$74.9645\\
				&    & 1.50 && 1.9469  & 1.1917  & 0.0060  &$-$91.6809 &&  1.5722 &  1.0720 & $-$0.9653 &$-$98.3310\\
				&    & 2.50 && 5.5552  &  1.5043 &   0.0030& $-$125.9847 &&   5.1288 &   1.3933 &  $-$0.7391 &$-$134.8456\\
				&    & 4.00 &&30.5304  &  3.0715 &   0.0040& $-$175.5051 &&  30.9898  &  2.9942 &  $-$0.5108 &$-$185.1701 \\                   
				\hline    
			\end{tabular}}
		\end{table}

		\begin{table}[ht!]
			\footnotesize
			\centering
			\caption{{Empirical mean from simulated MXBS data for the indicated model, estimator, generator, $p$ and $n$.}}
			\label{tab:MXBS}
			\renewcommand{\arraystretch}{1.0}
			\resizebox{\linewidth}{!}{
				\begin{tabular}{lrrrrrrrrrrrrrrrrrrrrrr}
					\hline
					Generator &    &          && \multicolumn{4}{c}{BBS} &&  \multicolumn{4}{c}{BBSO} \\\cline{5-8}\cline{10-13}\\ [-0.2cm]
					
					& $n$  &   $p$        &&   \multicolumn{1}{c}{$\widehat\alpha$} & \multicolumn{1}{c}{$\widehat\beta$} & \multicolumn{1}{c}{$\widehat\delta$} &  \multicolumn{1}{c}{log-lik} &&   \multicolumn{1}{c}{$\widehat\alpha$} & \multicolumn{1}{c}{$\widehat\beta$} & \multicolumn{1}{c}{$\widehat\gamma$} &  \multicolumn{1}{c}{log-lik}
					\\\hline\\ [-0.2cm]
					MXBS($\alpha_1=0.1,\beta_1=0.5,\alpha_2=1.0,\beta_2=2.0,p$)& 10 & 0.25 && 0.6078  & 1.8486 &  0.2890 &$-$14.8013 &&  0.4622 &  1.7921 & $-$1.6841 &$-$15.4262 \\
					&    & 0.50 && 0.5013  &  1.5308&   0.7390&$-$10.8898 &&  0.3790 &  1.4707 & $-$1.7584 &$-$11.5081 \\
					&    & 0.75 && 0.3556  &1.1169  &1.1480   &$-$4.3414  && 0.2639  &1.0798   &$-$1.8864  &$-$5.0188\\ [0.05cm]
					& 50 & 0.25 && 0.8350  & 1.9243 &  0.3120 &$-$85.1100 &&  0.5803 &  1.7684 & $-$1.2486 &$-$87.6034\\
					&    & 0.50 && 0.6721  & 1.6953 &  0.7320 &$-$65.5651 &&  0.4772 &  1.4628 & $-$1.3057 &$-$68.2009\\
					&    & 0.75 && 0.4502  & 1.2270 &  1.2940 &$-$32.7972 &&  0.3395 &  1.0801 & $-$1.3767 &$-$36.8795\\[0.1cm]
					%
					%7
					MXBS($\alpha_1=1.0,\beta_1=0.5,\alpha_2=0.5,\beta_2=5.0,p$)& 10 & 0.25 && 0.3295 &  4.0274 & $-$0.0520 &$-$17.7737&&   0.2517&   3.9848&  $-$1.6814& $-$18.4817  \\
					&    & 0.50 && 0.3136 &  2.9319 & $-$0.2210 &$-$14.1957&&   0.2419&   2.8967&  $-$1.6687& $-$14.9392\\ 
					&    & 0.75 && 0.3161 &  1.8058 & $-$0.2920 &$-$9.4978 &&  0.2434 &  1.7908 & $-$1.6703 &$-$10.2457\\[0.05cm]
					& 50 & 0.25 && 0.4509 &   4.0453&  0.0570   &$-$99.7776&&    0.3215&  3.9677 & $-$1.2011 & $-$102.9411  \\
					&    & 0.50 && 0.4309 &  2.9188 & $-$0.0040 &$-$81.8596&&   0.3080 &  2.8846 & $-$1.1967 &$-$85.1153\\
					&    & 0.75 && 0.4322 &  1.7842 & $-$0.0130 &$-$57.9865&&   0.3085 &  1.7790 & $-$1.2038 &$-$61.1831  \\[0.1cm]
					MXBS($\alpha_1=2.5,\beta_1=1.0,\alpha_2=0.5,\beta_2=1.0,p$)& 10 & 0.25 && 0.5487 &  1.7084 &  1.2980 &$-$11.7933&&   0.4194  & 1.6472 & $-$1.7132 &$-$12.6414 \\
					&    & 0.50 &&     0.7658 &  2.0200 &  1.9640 &$-$15.2106&&   0.5700 &  1.9589 & $-$1.8182 &$-$15.9973 \\
					&    & 0.75 &&     1.0252 &  1.9457 &  1.6240 &$-$17.0309&&  0.7539  & 1.8879  & $-$1.8284 &$-$17.6494 \\[0.05cm]
					& 50 & 0.25 && 0.7016 &  1.9251 &  0.9980& $-$68.7779&&   0.5432 &  1.6049 & $-$1.2276 &$-$74.1506 \\
					&    & 0.50 && 0.9403 &  2.2861 &  1.4130& $-$85.5513&&   0.7108 &  1.9387 & $-$1.4005 &$-$89.7587\\
					&    & 0.75 && 1.2874 &  2.1852 &  1.2490& $-$94.6959&&   0.9241 &  1.8571 & $-$1.4378 &$-$96.8759\\
					\hline    
				\end{tabular}}
			\end{table}
			
		}

		\newpage
		\section{Real data analysis}\label{sec:5}
		
		%{
		%% begin color red %%%%%%%%%%%%%%%%%%%%%%%%%%%%%%%%%%%%%%%%%%%%%%%%%%%%%%%%%%%%%%%%%%%%%%%%%%%%%%%%%%%%%%%%%%%%%%%%%%%%%%%%%%%%%%%% 
		
		The proposed BBS model is now used to analyse three lifetime data sets. 
		For comparison, the results of the bimodal BBSO model (bimodal BS distribution 
		proposed by \cite{omb:17}) and MXBS distribution introduced by \cite{bgkls:11}, in addition 
		to classical BS and LN models, are given as well.

		\begin{exmp}
			The first data set corresponds to the duration of the eruption for the Old Faithful geyser in 
			Yellowstone National Park, Wyoming, USA; see \cite{ab:90}. Descriptive statistics for the Old Faithful data set are 
			the following: $272$(sample size), $43$(minimum), $96$(maximum), $76$(median), $70.897$(mean), $13.595$(standard deviation), $19.176$(coefficient of variation), $-0.414$(coefficient of skewness) and $-1.156$(coefficient of kurtosis). 
			Table~\ref{tab:fit1} reports the ML estimates, computed by the BFGS method, SEs and log-likelihood (log-lik) values for the BBS, BBSO, MXBS, BS and LN models. 
			Furthermore, we report the Akaike (AIC) and Bayesian information (BIC) criteria. From this table, we note that the BBS and MXBS models provide better adjustments compared to the other models based on the values of AIC and BIC. The null hypothesis of a BS distribution ($\delta=0$) against an alternative 
			BBS distribution ($\delta\neq{0}$) can be tested by using the likelihood ratio (LR) test 
			$\textrm{LR}=-2(\ell_{\text{BS}}(\widehat\alpha,\widehat\beta)-\ell_{\text{BBS}}(\widehat\alpha,\widehat\beta,\widehat\delta))$. 
			In this case, we obtain $LR=-2(-1107.849+1050.592)=114.514$ and comparing it to the $5\%$ critical value from the chi-square distribution 
			with one degree of freedom ($\chi_{1}^{2}=3.84$), it supports rejection of the null hypothesis, thus 
			the BBS model outperforms, in terms of fitting, the BS one for the data under study.

			Figure~\ref{fig:3} shows the histogram of the data set superimposed with the fitted curves 
			of the BBS, BBSO, MXBS, BS and LN distributions. From this figure, we clearly note that the BBS captures quite well 
			the inherent bimodality of the data.

			\begin{table}[!ht]
				\centering
				\caption{ML estimates and model selection measures for fit to the Old Faithful data.}
				\label{tab:fit1}
				\renewcommand{\arraystretch}{1.3}
				\resizebox{\linewidth}{!}{
					\begin{tabular}{llcrcccccccccc}
						\hline
						Model     &              & Parameter     & \multicolumn{1}{c}{ML estimate}& SE       &   & log-lik      &   AIC      &   BIC    \\
						\hline\\[-0.25cm]
						BBS           &          & $\alpha$      &   0.1255                       &0.0034    &   &  $-$1050.592 & 2107.184   & 2118.001   \\[-0.2cm]
						&          & $\beta$       &   66.8612                      &0.4739    &   &              &            &            \\[-0.2cm]
						&          & $\delta$      &   $-$4                         &          &   &              &            &            \\[0.05cm]%\cline{2-8}
						BBSO          &          & $\alpha$      &   0.0893                       & 0.0047   &   &  $-$1054.396 & 2114.792   & 2125.609   \\[-0.2cm]
						&          & $\beta$       &   65.7730                      & 0.4128   &   &              &            &            \\[-0.2cm]
						&          & $\gamma$      &   $-$2.1803                    & 0.1432   &   &              &            &            \\[0.05cm]%\cline{2-8}
						MXBS          &          & $\alpha_{1}$  &   0.1150                       & 0.0046    &   &  $-$1032.681 & 2075.362   & 2093.391   \\[-0.2cm]
						&          & $\alpha_{2}$  &   0.0697                       & 0.0108    &   &              &            &            \\[-0.2cm]
						&          & $\beta_{1}$   &   54.8174                      & 0.4824    &   &              &            &            \\[-0.2cm]
						&          & $\beta_{2}$   &   80.1850                      & 0.7607    &   &              &            &            \\[-0.2cm]              
						&          & $p$           &   0.3762                       & 0.0317    &   &              &            &            \\[0.05cm]%\cline{2-8}                          
						BS            &          & $\alpha$      &   0.2055                       & 0.0088   &   & $-$1107.849  & 2221.698   & 2226.91   \\[-0.2cm]
						&          & $\beta$       &   69.4289                      & 0.8608   &   &              &            &            \\[0.05cm]%\cline{2-8}
						LN            &          & $\mu$         &   4.2411                       & 0.0124   &   &  $-$1108.300 & 2222.6     & 2227.812   \\[-0.2cm]
						&          & $\sigma$      &   0.2048                       & 0.0087   &   &              &            &            \\\hline\\[-0.20cm]%\cline{2-8}
					\end{tabular}
				}
			\end{table}

			\begin{figure}[H]
				%\vspace{-0.25cm}
				\centering
				\psfrag{0.00}[c][c]{\scriptsize{0.00}}
				\psfrag{0.01}[c][c]{\scriptsize{0.01}}
				\psfrag{0.02}[c][c]{\scriptsize{0.02}}
				\psfrag{0.03}[c][c]{\scriptsize{0.03}}
				\psfrag{0.04}[c][c]{\scriptsize{0.04}}
				\psfrag{0.06}[c][c]{\scriptsize{0.06}}
				\psfrag{0.07}[c][c]{\scriptsize{0.07}}
				\psfrag{0.05}[c][c]{\scriptsize{0.05}}

				\psfrag{40}[c][c]{\scriptsize{40}}
				\psfrag{50}[c][c]{\scriptsize{50}}
				\psfrag{60}[c][c]{\scriptsize{60}}
				\psfrag{70}[c][c]{\scriptsize{70}}
				\psfrag{80}[c][c]{\scriptsize{80}}
				\psfrag{90}[c][c]{\scriptsize{90}}
				\psfrag{100}[c][c]{\scriptsize{100}}
				
				\psfrag{de}[c][c]{\scriptsize{PDF}}
				\psfrag{db}[c][c]{\scriptsize{waiting times}}
				
				\psfrag{a}[l][c]{\tiny{BBS}}
				\psfrag{b}[l][c]{\tiny{BBSO}}
				\psfrag{c}[l][c]{\tiny{MXBS}}
				\psfrag{d}[l][c]{\tiny{BS}}
				\psfrag{e}[l][c]{\tiny{LN}}
				
				{\includegraphics[height=8.5cm,width=5.0cm,angle=-90]{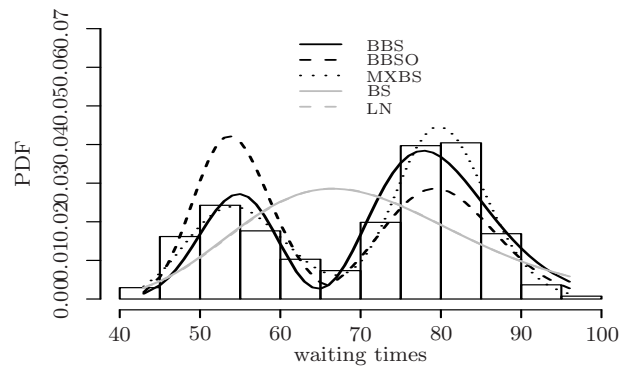}}
				\caption{Histogram of waiting times until the next eruption (from the Old Faithful data) 
					overlaid with the fitted densities.}\label{fig:3}
				\vspace{-0.45cm}
			\end{figure}
		\end{exmp}
		\begin{exmp}
			The data used here, which are given by \cite{ah:85} who attribute
			them to a study by \cite{btf:84}, present the stress rapture life in hours of Kevlar-49/epoxy
			strands when subjected to a constant sustained pressure until failure. A
			descriptive summary for the Kevlar-49/epoxy data set provides the following values: $49$(sample size), $1051$(minimum), 
			$17568$(maximum), $8831$(median), $8805.694$(mean), $4553.915$ (standard deviation), $51.176$(coefficient of variation), $0.094$(coefficient of skewness) and $-0.915$(coefficient of kurtosis).

			Table~\ref{tab:fit2} reports the ML estimates, SEs and log-lik values associated with the BBS, BBSO, MXBS, 
			BS and LN models. Furthermore, we report the values of AIC and BIC. From this table, observe that the proposed BBS model has the lowest values for the AIC 
			and BIC, suggesting that this model provides 
			the best fit to Kevlar-49/epoxy data. To test the null 
			hypothesis of a BS distribution ($\delta=0$) against an alternative 
			BBS distribution ($\delta\neq{0}$), we use the LR test. The result $LR=-2(-488.4345+480.049)=16.771$ 
			supports the BBS model assumption and rejects the BS model 
			for this data set. This result suggest that the BBS distribution is indeed a good model for the 
			Kevlar-49/epoxy data. A graphical comparison of the fitted BBS, BBSO, MXBS, BS and LN
			distributions is given in Figure \ref{fig:4}.

			\begin{table}[htbp]
				\centering
				\caption{\small {ML estimates and model selection measures for fit to the Kevlar-49/epoxy data.}}
				\label{tab:fit2}
				\renewcommand{\arraystretch}{1.3}
				\resizebox{\linewidth}{!}{
					\begin{tabular}{llcrrccccccccc}
						\hline
						Model     &              & Parameter     & \multicolumn{1}{c}{ML estimate} & SE        &  & log-lik      &   AIC      &   BIC    \\
						\hline\\[-0.25cm]
						BBS           &          & $\alpha$      &   0.5933                        & 0.0504    & &$-$480.049    & 966.098    & 971.773   \\[-0.2cm]
						&          & $\beta$       &  4507.365                       & 376.7557  & &              &            &            \\[-0.2cm]
						&          & $\delta$      &   $-$2                          &           & &              &            &            \\[0.05cm]%\cline{2-8}
						BBSO          &          & $\alpha$      &   0.4679                        & 0.08655   & &$-$490.208   & 986.417    & 992.092   \\[-0.2cm]
						&          & $\beta$       &  5036.113                       & 437.0908  & &             &            &            \\[-0.2cm]
						&          & $\gamma$      &   $-$1.5345                     & 0.4708    & &             &            &            \\[0.05cm]%\cline{2-8}
						MXBS          &          & $\alpha_{1}$  &   0.3412                       & 0.0166      & &$-$593.987   & 1197.974   & 1207.433   \\[-0.2cm]
						&          & $\alpha_{2}$  &   0.0697                       & 0.0123      & &             &            &            \\[-0.2cm]
						&          & $\beta_{1}$   &   10278.38                     & 979.1201    & &             &            &            \\[-0.2cm]
						&          & $\beta_{2}$   &   4577.872                     & 50.2859     & &             &            &            \\[-0.2cm]              
						&          & $p$           &   0.6708                       & 0.0991      & &             &            &            \\[0.05cm]%\cline{2-8}  
						BS            &          & $\alpha$      &   0.7520                        & 0.0759    &  &$-$488.434   & 982.869    & 984.652   \\[-0.2cm]
						&          & $\beta$       &   6800.546                      & 679.8509  &  &             &            &            \\[0.05cm]%\cline{2-8}
						LN            &          & $\mu$         &   8.8925                        & 0.1001    &  &$-$487.873   & 981.746    & 983.530   \\[-0.2cm]
						&          & $\sigma$      &   0.7012                        & 0.0708    &  &             &            &            \\\hline\\[-0.20cm]%\cline{2-8}
					\end{tabular}
				}
			\end{table}

			\begin{figure}[htbp]
				\vspace{-0.25cm}
				\centering
				\psfrag{0.00000}[c][c]{\scriptsize{0}}
				\psfrag{0.00005}[c][c]{\scriptsize{5e-5}}
				\psfrag{0.00010}[c][c]{\scriptsize{10e-5}}
				\psfrag{0.00015}[c][c]{\scriptsize{15e-5}}
				\psfrag{0.00020}[c][c]{\scriptsize{20e-5}}
				\psfrag{0.00025}[c][c]{\scriptsize{25e-5}}
				\psfrag{5000}[c][c]{\scriptsize{5000}}
				\psfrag{10000}[c][c]{\scriptsize{10000}}
				\psfrag{15000}[c][c]{\scriptsize{15000}}
				\psfrag{0}[c][c]{\scriptsize{0}}
				\psfrag{1}[c][c]{\scriptsize{1}}
				\psfrag{2}[c][c]{\scriptsize{2}}
				\psfrag{3}[c][c]{\scriptsize{3}}
				\psfrag{4}[c][c]{\scriptsize{4}}
				\psfrag{5}[c][c]{\scriptsize{5}}
				\psfrag{6}[c][c]{\scriptsize{6}}
				\psfrag{8}[c][c]{\scriptsize{8}}
				\psfrag{10}[c][c]{\scriptsize{10}}
				\psfrag{12}[c][c]{\scriptsize{12}}
				\psfrag{15}[c][c]{\scriptsize{15}}
				\psfrag{20}[c][c]{\scriptsize{20}}
				\psfrag{25}[c][c]{\scriptsize{25}}
				\psfrag{30}[c][c]{\scriptsize{30}}
				\psfrag{0.00}[c][c]{\scriptsize{0.00}}
				\psfrag{0.05}[c][c]{\scriptsize{0.05}}
				\psfrag{0.10}[c][c]{\scriptsize{0.10}}
				\psfrag{0.15}[c][c]{\scriptsize{0.15}}
				\psfrag{0.20}[c][c]{\scriptsize{0.20}}
				\psfrag{0.25}[c][c]{\scriptsize{0.25}}
				\psfrag{0.30}[c][c]{\scriptsize{0.30}}
				\psfrag{0.0}[c][c]{\scriptsize{0.0}}
				\psfrag{0.2}[c][c]{\scriptsize{0.2}}
				\psfrag{0.3}[c][c]{\scriptsize{0.3}}
				\psfrag{0.4}[c][c]{\scriptsize{0.4}}
				\psfrag{0.5}[c][c]{\scriptsize{0.5}}
				\psfrag{0.6}[c][c]{\scriptsize{0.6}}
				\psfrag{0.8}[c][c]{\scriptsize{0.8}}
				\psfrag{1.0}[c][c]{\scriptsize{1.0}}
				\psfrag{1.2}[c][c]{\scriptsize{1.2}}
				\psfrag{1.5}[c][c]{\scriptsize{1.5}}
				\psfrag{2.0}[c][c]{\scriptsize{2.0}}
				\psfrag{2.5}[c][c]{\scriptsize{2.5}}
				\psfrag{3.0}[c][c]{\scriptsize{3.0}}
				\psfrag{de}[c][c]{\scriptsize{PDF}}
				\psfrag{db}[c][c]{\scriptsize{failure times}}
				
				\psfrag{a}[l][c]{\tiny{BBS}}
				\psfrag{b}[l][c]{\tiny{BBSO}}
				\psfrag{c}[l][c]{\tiny{MXBS}}
				\psfrag{d}[l][c]{\tiny{BS}}
				\psfrag{e}[l][c]{\tiny{LN}}
				
				{\includegraphics[height=8.5cm,width=5.0cm,angle=-90]{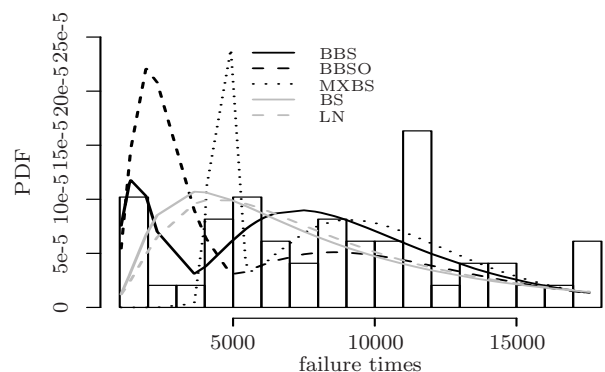}}
				\caption{Histogram of Kevlar 49/epoxy strands failure times 
					(70\% pressure) overlaid with the fitted densities.}\label{fig:4}
				\vspace{-0.45cm}
			\end{figure}

		\end{exmp}
		
		\begin{exmp}
			The third data set corresponds to the lifetimes of adult flies in days after exposure to a pest control technique, which consists of 
			using small portions of food laced with an insecticide that kills the flies. The experiment was carried out at the Department of Entomology of the Luiz
			de Queiroz School of Agriculture, University of S\~ao Paulo, Brazil. In this technique, the period was set at 51 days
			such that larvae that survived beyond this period are considered as censored cases; see \cite{setal:13} for more details about 
			this experiment. Descriptive statistics for the Entomology data are 
			the following: $\textrm{sample size} = 172$ (four cases are lost), $\textrm{minimum} = 1.000$, $\textrm{maximum} = 51.000$, 
			$\textrm{median} = 21.000$, $\textrm{mean} = 21.878$, $\textrm{standard deviation} = 11.674$, 
			$\textrm{coefficient of variation} = 53.30$, $\textrm{coefficient of skewness} = 0.818$ and 
			$\textrm{coefficient of kurtosis} = 0.569$. 
			
			The ML estimates and log-lik values for the BBS, BBSO, MXBS, BS and LN models are reported in 
			Table~\ref{tab:fit3}. Furthermore, the AIC and BIC values 
			are also reported in this table. From Table~\ref{tab:fit3}, 
			we note that the proposed BBS model has the lowest AIC and BIC values, and therefore it could be chosen 
			as the best model. Using the LR statistic to compare the fits of the BS and BBS models, 
			that is, the null hypothesis of a BS distribution ($\delta=0$) against an alternative 
			BBS distribution ($\delta\neq{0}$), we obtain $LR=-2(-676.913+610.523)=132.780$ 
			and then we could accept the BBS model. Figure \ref{fig:5} 
			shows the fitted PDFs and SFs (by Kaplan-Meier (KM) estimator) of the BBS, BBSO, BS and LN distributions.

			\begin{table}[H]
				\centering
				\caption{\small {ML estimates and model selection measures for fit to the Entomology data.}}
				\label{tab:fit3}
				\renewcommand{\arraystretch}{1.3}
				\resizebox{\linewidth}{!}{
					\begin{tabular}{llcrrccccccccc}
						\hline
						Model     &              & Parameter  & \multicolumn{1}{c}{ML estimate}& SE     & log-lik      & AIC        & BIC    \\
						\hline\\[-0.25cm]
						BBS           &          & $\alpha$   &  0.6922                        & 0.0385 & $-$610.523   & 1227.052   & 1236.494   \\[-0.2cm]
						&          & $\beta$    &  8.7636                        & 0.5311 &              &            &            \\[-0.2cm]
						&          & $\delta$   &  $-$2                          &        &              &            &            \\[0.05cm]%\cline{2-8}
						BBSO          &          & $\alpha$   &  0.4975                        & 0.0437 & $-$663.144   & 1332.287   & 1341.730   \\[-0.2cm]
						&          & $\beta$    &  7.5959                        & 0.5250 &              &            &            \\[-0.2cm]
						&          & $\gamma$   &  $-$2.2434                     & 0.2859 &              &            &            \\ [0.05cm]%\cline{2-8}
						MXBS          &          & $\alpha_{1}$&   1.2349                      & 0.0115 & $-$631.137   & 1266.273   & 1272.568   \\[-0.2cm]
						&          & $\alpha_{2}$&   0.2104                      & 0.0009 &              &            &            \\[-0.2cm]
						&          & $\beta_{1}$ &   14.6760                     & 2.7919 &              &            &            \\[-0.2cm]
						&          & $\beta_{2}$ &   19.9960                     & 0.4717 &              &            &            \\[-0.2cm]              
						&          & $p$         &   0.60                        &        &              &            &            \\[0.05cm]%\cline{2-8} 
						BS            &          & $\alpha$    &   0.8912                      & 0.0500 & $-$676.913   & 1357.862   & 1364.157   \\[-0.2cm]
						&          & $\beta$     &   16.1512                     & 1.0078 &              &            &            \\[0.20cm]%\cline{2-8}
						LN            &          & $\mu$       &   2.9139                      & 0.0583 & $-$660.230   & 1324.461   & 1330.756   \\[-0.2cm]
						&          & $\sigma$    &   0.7613                      & 0.0708 &              &            &            \\\hline
					\end{tabular}
				}
			\end{table}

			\begin{figure}[H]
				\vspace{-0.25cm}
				\centering
				
				\psfrag{0}[c][c]{\scriptsize{0}}
				\psfrag{10}[c][c]{\scriptsize{10}}
				\psfrag{20}[c][c]{\scriptsize{20}}
				\psfrag{30}[c][c]{\scriptsize{30}}
				\psfrag{40}[c][c]{\scriptsize{40}}
				\psfrag{50}[c][c]{\scriptsize{50}}
				\psfrag{0.00}[c][c]{\scriptsize{0.00}}
				\psfrag{0.05}[c][c]{\scriptsize{0.05}}
				\psfrag{0.10}[c][c]{\scriptsize{0.10}}
				\psfrag{0.15}[c][c]{\scriptsize{0.15}}
				\psfrag{0.20}[c][c]{\scriptsize{0.20}}
				\psfrag{0.25}[c][c]{\scriptsize{0.25}}
				\psfrag{0.2}[c][c]{\scriptsize{0.20}}
				\psfrag{0.4}[c][c]{\scriptsize{0.40}}
				\psfrag{0.6}[c][c]{\scriptsize{0.60}}
				\psfrag{0.8}[c][c]{\scriptsize{0.80}}
				\psfrag{1.0}[c][c]{\scriptsize{1.00}}
				\psfrag{de}[c][c]{\scriptsize{PDF}}
				\psfrag{db}[c][c]{\scriptsize{survival times}}
				\psfrag{st1}[c][c]{\scriptsize{SF}}
				\psfrag{t1}[c][c]{\scriptsize{survival times}}
				\psfrag{km}[l][c]{\tiny{KM}}
				\psfrag{a}[l][c]{\tiny{BBS}}
				\psfrag{b}[l][c]{\tiny{BBSO}}
				\psfrag{c}[l][c]{\tiny{MXBS}}
				\psfrag{d}[l][c]{\tiny{BS}}
				\psfrag{e}[l][c]{\tiny{LN}}
				
				{\includegraphics[height=5.5cm,width=5.5cm,angle=-90]{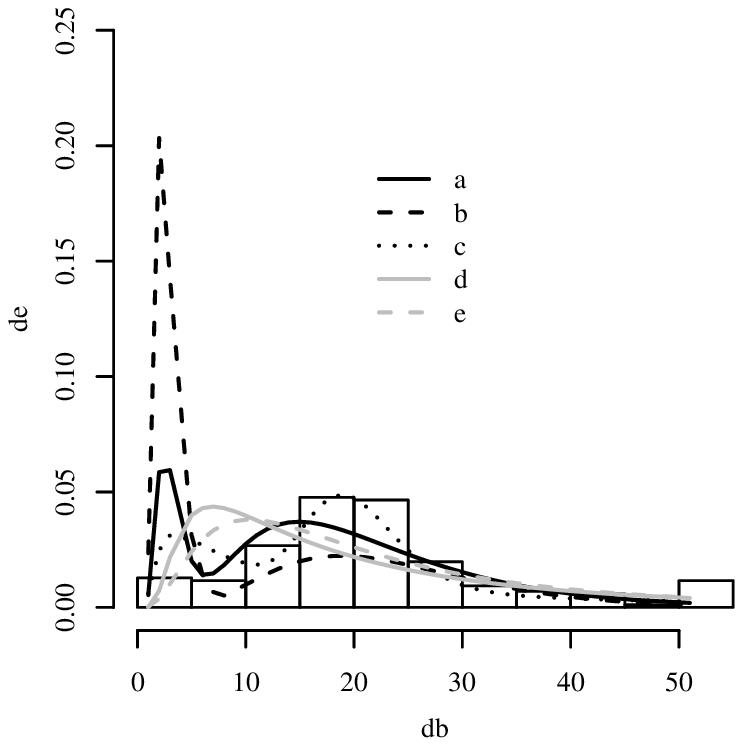}}
				{\includegraphics[height=5.5cm,width=5.5cm,angle=-90]{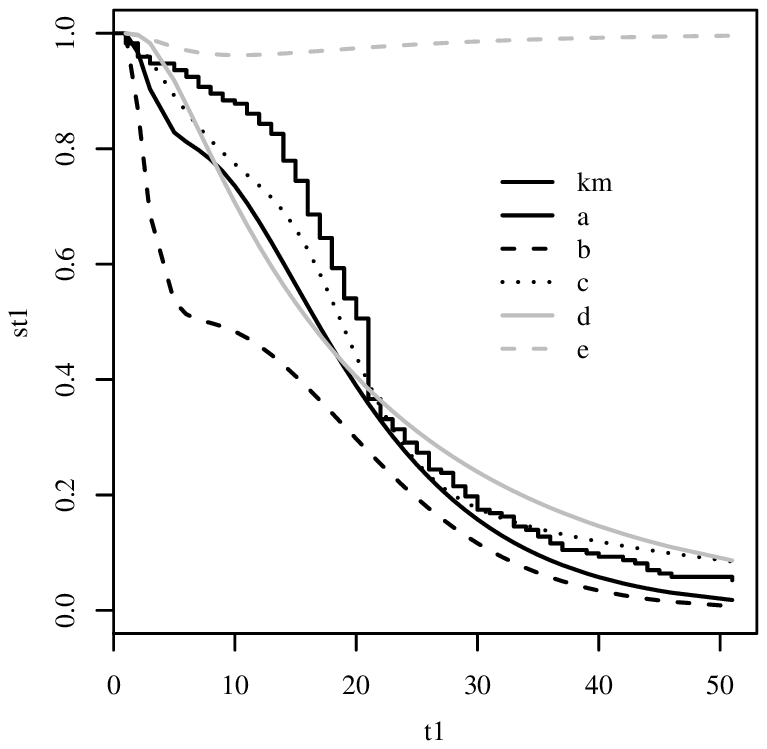}}
				\caption{Histogram and SF fitted by KM with the Entomology data.}\label{fig:5}
				\vspace{-0.45cm}
			\end{figure}

		\end{exmp}

		%}
		%% begin color red %%%%%%%%%%%%%%%%%%%%%%%%%%%%%%%%%%%%%%%%%%%%%%%%%%%%%%%%%%%%%%%%%%%%%%%%%%%%%%%%%%%%%%%%%%%%%%%%%%%%%%%%%%%%%%%%

		\section{Concluding remarks}\label{sec:6}
		%{
		%% begin color red %%%%%%%%%%%%%%%%%%%%%%%%%%%%%%%%%%%%%%%%%%%%%%%%%%%%%%%%%%%%%%%%%%%%%%%%%%%%%%%%%%%%%%%%%%%%%%%%%%%%%%%%%%%%%%%%
		In this work, we have introduced a bimodal generalization 
		of the Birnbaum-Saunders distribution, based on the 
		alpha-skew-normal distribution. We have discussed some of its properties. 
		We have considered estimation and inference based on likelihood methods. We have carried out a Monte Carlo simulation study to evaluate the behavior of the 
		maximum likelihood estimators of the corresponding parameters. Three real data sets were considered to illustrate the potentiality of the proposed model. 
		In general, the results have shown that the proposed bimodal Birnbaum-Saunders distribution outperforms some existing models in the literature. As part of future research, 
		it is of interest to study univariate and multivariate bimodal Birnbaum-Saunders regression models; see \cite{rn:91}, \cite{bz:15} and \cite{mlc:16}. Moreover, time 
		series models based on the bimodal Birnbaum-Saunders distribution with corresponding influence diagnostic tools can also be considered; see \cite{slla:17}. 
		Work on these problems is currently under progress and we hope to report these findings in a future paper.
		%}
		%% end color red %%%%%%%%%%%%%%%%%%%%%%%%%%%%%%%%%%%%%%%%%%%%%%%%%%%%%%%%%%%%%%%%%%%%%%%%%%%%%%%%%%%%%%%%%%%%%%%%%%%%%%%%%%%%%%%%
		\section*{Acknowledgments}
		The authors thank the Editors and reviewers for their constructive comments on an
		earlier version of this manuscript. The research was partially supported by CNPq and CAPES grants from the Brazilian government.

\end{document}